\documentclass{article}
\usepackage{amssymb}
\usepackage{amsmath,color}
\usepackage[hidelinks]{hyperref}

\usepackage[square,numbers]{natbib}         
\usepackage{ntheorem}
\newtheorem{lem}{Lemma}[section]
\newtheorem{cor}[lem]{Corollary}
\newtheorem{teo}[lem]{Theorem}

\newtheorem{os}[lem]{Remark}
\newtheorem{defi}[lem]{Definition}
\newtheorem{prop}[lem]{Proposition}

\newenvironment{proof}{{\sc{Proof.}}}{\hfill\qed}

\newcommand{\qed}{\thinspace\null\nobreak\hfill\hbox{\vbox{\kern-.2pt\hrule
			height.2pt depth.2pt\kern-.2pt\kern-.2pt \hbox to2.5mm{\kern-.2pt\vrule
				width.4pt \kern-.2pt\raise2.5mm\vbox to.2pt{}\lower0pt\vtop
				to.2pt{}\hfil\kern-.2pt \vrule
				width.4pt \kern-.2pt}\kern-.2pt\kern-.2pt\hrule height.2pt depth.2pt
			\kern-.2pt}}\par\medbreak}

\newcommand{\R}{\mathbb{R}}

\newcommand{\C}{\mathbb{C}}

\newcommand{\N}{\mathbb{N}}

\newcommand{\Rp}{\textrm{\emph{Re}\,}}

\newcommand{\eps}{\varepsilon}

\newcommand{\ov}{\overline}

\newcommand{\ds}{\displaystyle}

\date{}
\begin{document}

\title{
Sharp kernel bounds  for  parabolic operators with first order degeneracy}
\author{L. Negro \thanks{Dipartimento di Matematica e Fisica``Ennio De Giorgi'', Universit\`a del Salento, C.P.193, 73100, Lecce, Italy.
e-mail:  luigi.negro@unisalento.it} \qquad C. Spina \thanks{Dipartimento di Matematica e Fisica ``Ennio De Giorgi'', Universit\`a del Salento, C.P.193, 73100, Lecce, Italy.
e-mail:  chiara.spina@unisalento.it}}

\maketitle
\begin{abstract}
\noindent 
We prove sharp upper and lower estimates for the parabolic kernel of  the singular elliptic operator
\begin{align*}
	\mathcal L&=\mbox{Tr }\left(AD^2\right)+\frac{\left(v,\nabla\right)}y,
\end{align*}
in the half-space $\R^{N+1}_+=\{(x,y): x \in \R^N, y>0\}$ under Neumann or oblique derivative boundary conditions at $y=0$.

\par

\end{abstract}

\section{Introduction}
In this paper   we  study sharp upper and lower estimates for the parabolic kernel of the   singular   operator
 \begin{equation}\label{def L}
\mathcal L=\mbox{Tr }\left(AD^2\right)+\frac{\left(v,\nabla\right)}y
=\sum_{i,j=1}^{N+1}a_{ij}D_{ij}+\frac{ d\cdot \nabla_x+cD_y }{y}
\end{equation}
in the half-space $\R^{N+1}_+=\{(x,y): x \in \R^N, y>0\}$.  Here   $A=\left(a_{ij}\right)\in \R^{N+1,N+1}$ is a symmetric and positive definite 
matrix and   we suppose $\frac c\gamma +1>0$, where $\gamma=a_{N+1,N+1}$; the vector $v=(d,c)\in\R^{N+1}$ satisfies  $d=0$ if $c=0$ i.e. $v$ is  oblique with respect to the boundary of $\R^{N+1}_+$.  We    endow $\mathcal L$ with   Neumann or  oblique derivative boundary conditions at $y=0$
$$\lim_{y\to 0} D_y u=0\quad \text{(if $v=0$)},\qquad\qquad  \lim_{y\to 0}y^{\frac c\gamma}\, v \cdot \nabla u=0\quad \text{(if $c\neq 0$)}.$$
Up to some suitable linear transformations, the latter operator is equivalent to  the model operator 
\begin{align}\label{intro model operator}
	\mathcal L =\Delta_{x} +2a\cdot\nabla_xD_y+ B_y, \qquad  B_y=D_{yy}+ \frac cy D_y
\end{align}
were $B_y$ is a one-dimensional Bessel operator, $c+1>0$  and  $a\in\R^N$ satisfies the requirement $|a| <1$ which is equivalent to the ellipticity of the top order coefficients.

\medskip
In the special case $a=0$ and  $\mathcal L=\Delta_x+B_y$,  that is when the mixed second order derivatives do not appear, the   Bessel operator $B_y$ and the Laplacian $\Delta_x$ commute and sharp kernel estimates  can be  proved  after an explicit description of the kernel of the Bessel operator (see \cite{MNS-Caffarelli}).   These operators  play a major role in the investigation of the fractional powers $(-\Delta_x)^s$ and  $(D_t-\Delta_x)^s$, $s=(1-c)/2$, through the  ``extension procedure" of Caffarelli and Silvestre, see \cite{Caffarelli-Silvestre}.

\medskip 
Unfortunately  adding the mixed second order derivatives term is an hard complication of the problem, 
one important reason being the loss of commutativity. However the study of the general operator \eqref{def L}   is   crucial  for treating degenerate operators in domains since mixed derivatives and oblique boundary conditions appear in  the localization procedure.
For this reason its properties have been  studied in a series of previous papers. Elliptic and parabolic solvability of the associated problems in weighted $L^p$ spaces have been  investigated in \cite{MNS-Singular-Half-Space, MNS-Degenerate-Half-Space, MNS-Caffarelli, MNS-CompleteDegenerate, Negro-AlphaDirichlet} where we proved, under suitable assumptions on $m$ and $p$ and for $\frac{c}{\gamma}+1>0$, that  $\mathcal L$ generates an analytic semigroup in $L^p_m$ $=L^p(\R^{N+1}_+; y^m dxdy)$.  In addition  we characterized its domain as a weighted Sobolev space and showed that it has the  maximal regularity.

\medskip
In addition we also refer to   \cite{dong2020parabolic, dong2021weighted, dong2020RMI, dong2020neumann} where the authors studied  operators of this form, even for variable coefficients, using tools form linear PDE and Muckhenoupt weights. In \cite{Robinson-Sikora2008}, among other properties, kernel estimates for some second order degenerate operators in divergence form are considered.

\medskip
Our main results are Theorems  \ref{lower} and \ref{complete-sharp}  where   we prove that $\mathcal L$ generates in $L^2_{\frac c\gamma}$ a contractive analytic semigroup $\left(e^{t\mathcal L}\right)_{t\in\Sigma}$  on some  suitable sector  $\Sigma$ of the complex plane,	and its heat kernel $p_{{\mathcal L}}$, written  with respect the measure $y^\frac{c}{\gamma}dz$, satisfies the following two-sided  estimates 
\begin{align}\label{Intro est uplow}
	p_{{\mathcal  L}}(t,z_1,z_2)
	\simeq C t^{-\frac{N+1}{2}} y_1^{-\frac{c}{2\gamma}} \left(1\wedge \frac {y_1}{\sqrt t}\right)^{\frac{c}{2\gamma}} y_2^{-\frac{c}{2\gamma}} \left(1\wedge \frac{y_2}{\sqrt t}\right)^{\frac{c}{2\gamma}}\,\exp\left(-\dfrac{|z_1-z_2|^2}{kt}\right)
\end{align}
 where $t>0$, $z_1=(x_1,y_1),\ z_2=(x_2,y_2)\in\R^{N+1}_+$ and $C,k>0$ are some positive constants which may be different in the lower and upper bounds. We also point out here that the above inequalities can be written in the equivalent form \eqref{upper estimates ver2} (see Remark \ref{oss equiv estimate}). 
 
We    proved   the upper  bound  of  \eqref{Intro est uplow}   in \cite{Negro-Spina-SingularKernel}   by  revisiting the classical method based upon the equivalence between Gagliardo-Nirenberg type inequalities and ultra-contractivity estimates; we also proved that  the same estimate extends  to complex times  $t\in\Sigma$.  
 
 In this  the paper we   therefore focus on proving the    lower bounds in  \eqref{Intro est uplow}.
As a needed, but also of independent interest, result,   we prove  in Theorem \ref{space derivative estimates}  the following  gradient estimates 
	\begin{align} \label{Intro est grad}
			\left|\nabla p_\mathcal L(t,z_1,z_2)\right|
			\	\leq C |t|^{-\frac{N+2}{2}}  y_2^{-\frac c\gamma } \left(1\wedge \frac{y_2}{\sqrt{ |t|}}\right)^{\frac c\gamma }\,\exp\left(-\dfrac{|z_1-z_2|^2}{k|t|}\right),\qquad t\in\Sigma.
 \end{align}

\medskip
Let us describe the  strategy of proof employed. Let us discuss, preliminarily and  in an abstract setting, a  well established  strategy (see e.g. \cite[Chapter 7, Section 7.8]{Ouhabaz}) for deducing lower heat kernel estimates referring to the discussion at  the beginning of Section \ref{section kernel lower}  for further details.  Let $A$ be  a non-negative self-adjoint operator  acting  on  $L^2\left(X,\mu\right)$ where  $X$ is a homogeneous space i.e. a metric space $(X,d)$ endowed with a doubling measure $\mu$. Assume that the semigroup $e^{tA}$ has a kernel $p(t,z_1,z_2)$ which satisfies  the Gaussian upper bound
	\begin{align}\label{abstract up kernel}
	p(t,z_1,z_2)
	&\leq \frac{C}{V\left(z_1,\sqrt t\right)^{\frac 1 2}V\left(z_2,\sqrt t\right)^{\frac 1 2}}\exp\left(-\frac{d(z_1,z_2)^2}{\kappa t}\right),\quad \forall \ t>0, \ z_1,\ z_2\in X.
\end{align} where $V(z,r)=\mu(B(z,r))$ is the measure of the $X$-balls. Then the first crucial step consists  in the derivation of the  on-diagonal lower  estimate 
\begin{align}\label{abstract lower  diagonal}
	p(t,z,z)\geq \frac{C}{V(z,\sqrt t)},\qquad z\in X,\ t>0,
\end{align}
which can be proved by  combining the conservation  property of the kernel and the upper estimate \eqref{abstract up kernel} (see e.g. \cite[Proposition 7.28]{Ouhabaz}). The general off-diagonal lower bound then is proved by using the H\"older continuity of $p$ to go outside of the diagonal and an iterative argument based on the validity of the so-called chain condition  for $X$.\\

In our case, due to the non self-adjointness of $\mathcal L$, this strategy does not work. Then, inspired by the method of Nash \cite{Nash,Fabes-Stroock}, we  prove  the on-diagonal lower estimate by proving, in Proposition \ref{log-kernel-est}, the so-called  ``G-bound''
\begin{align}\label{Gbound intro}
	\int_{\R^{N+1}_+} \log p_{\mathcal L}(1,z,z_2)e^{-\alpha|z_2|^2}y_2^\frac c\gamma \, dz_2\geq -C,\qquad |z|\leq 1,
\end{align}
namely a lower bound of the integral of the logarithm of the heat kernel  with respect to the    weighted gaussian measure $\nu=e^{-\alpha|z|^2}y^\frac c\gamma \, dz$, $\alpha>0$.  This is not an easy task due to the presence of the weight $y^\frac{c}{\gamma}$  introduced to treat the degeneracy of the operator and which can be either singular or degenerate at $y=0$ and at infinity.  We emphasize that the  derivation of  \eqref{Gbound intro} relies on the  validity of some Poincar\'{e}-type inequalities in $H^1_\nu(\R^{N+1}_+)$ which require a separate study and which we prove in \cite{Negro-Spina-Poincare}. 

Another  difficulty appears: in order to compensate the non self-adjointness of $\mathcal L$, we need also  to  study in parallel both $\mathcal L$ and $\mathcal L^\ast$. In particular once \eqref{Gbound intro} is proved for both $\mathcal L$ and $\mathcal L^\ast$, then  by using  the reproducing property of the kernel and  Jensen inequality we derive, in Proposition \ref{lower-diag}, the on-diagonal lower estimate for  $p_{\mathcal L}$. Finally, thanks to the gradient estimates  proved in Section \ref{section gradient estimate}, we can go outside the diagonal proving the off-diagonal lower bound in  \eqref{Intro est uplow}.\\

\medskip

We now describe the structure of the paper.

\medskip
In  Section \ref{section general operator}  we introduce the operator $\mathcal L$ in \eqref{def L} and we collect the  results we need about generation of semigroups, maximal regularity and domain characterization in the weighted spaces $L^p_m(\R_+^{N+1}):=L^p(\R_+^{N+1}; y^m dx dy)$  proved in \cite{MNS-Singular-Half-Space, Negro-AlphaDirichlet}. We also  recall the upper kernel estimates   proved  in \cite{Negro-Spina-SingularKernel}.

\medskip

In  Section \ref{Section oblique}  we describe a simple change of variable which  allows to simplify the discussion to the model operator \eqref{intro model operator} under Neumann boundary conditions.  In all the subsequent sections  we therefore deal, without any loss of generality, almost exclusively with the model case, leaving to Theorem \ref{complete-sharp}  the  generalization of our results to the general case. 

\medskip

In Section \ref{L2} we write the forms associated with $\mathcal L$ and the adjoint $\mathcal L^*$  in the weighted space  $L^2_c$, where the weighted measure  $y^c\,dz$ takes into account the degeneracy of the operators. We state the main properties of  the forms which allow to get generation of analytic semigroups in $L^2_c$. 

\medskip
In Section \ref{section gradient estimate} we prove gradient estimates for the heat kernels of $\mathcal L$ and $\mathcal L^\ast$. The strategy here relies on the analyticity of the semigroup and the Cauchy formula for holomorphic functions which allow to deduce a pointwise estimates for $D_t p_{\mathcal L}=\mathcal L p_{\mathcal L}$. Then the estimate for the gradient is proved by an interpolation argument  using  the solvability theory of Section \ref{section general operator} for $\mathcal L$  in $L^p\left(\R_+^{N+1}\right)$  for $p>N+1$ and the  Morrey's embeddings. 

\medskip

In Section \ref{Section conservativity} we prove that the semigroups $e^{t\mathcal L}$, $e^{t\mathcal L^\ast }$, initially defined on $L^2_c$, extend also  on $L^1_c$ and we prove their  conservation property.

\medskip

Section \ref{section kernel lower}, which contains the main result of the paper, is devoted to prove the lower kernel  estimates. As explained before the crucial step consists in proving the   ``G-bound'' \eqref{Gbound intro} which allows to deduce  the lower bounds on the diagonal. Here we mainly focus on the range $\frac{|x|}{\sqrt t}, \frac{y}{\sqrt t}\leq 1$  while the lower bound far away from the boundary is deduced  by domination with a uniformly elliptic operator.

\medskip
In Appendix  we  characterize the boundedness  on  $L^p_m$ of a family of  integral operators which are related to our operator.

\bigskip
\noindent\textbf{Notation.} For $N \ge 0$, $\R^{N+1}_+=\{(x,y): x \in \R^N, y>0\}$. We write $\nabla u, D^2 u$ for the gradient and the Hessian matrix of a function $u$ with respect to all $x,y$ variables and $\nabla_x u, D_y u, D_{x_ix_j }u, D_{x_i y} u$ and so on, to distinguish the role of $x$ and $y$.  For $m \in \R$ we consider the measure $y^m dx dy $ in $\R^{N+1}_+$ and  we write $L^p_m(\R_+^{N+1})$, and often only $L^p_m$ when $\R^{N+1}_+$ is understood, for  $L^p(\R_+^{N+1}; y^m dx dy)$. 
$\C^+$ stands for $\{ \lambda \in \C: \Rp \lambda >0 \}$ and, for $|\theta| \leq \pi$, we denote by  $\Sigma_{\theta}$  the open sector $\{\lambda \in \C: \lambda \neq 0, \ |Arg (\lambda)| <\theta\}$.  

We write often $(x,y)$ or $x\cdot y$ to denote the inner product of $\R^N$ and, for $A,B\in\R^{N,N}$ symmetric, $\mbox{Tr }\left(AB\right)=\sum_{i,j}a_{ij}b_{i,j}$. 
Given $a$ and $b$ $\in\R$, $a\wedge b$, $a \vee b$  denote  their minimum and  maximum. We  write $f(x)\simeq g(x)$ for $x$ in a set $I$ and positive $f,g$, if for some $C_1,C_2>0$ 
\begin{equation*}
	C_1\,g(x)\leq f(x)\leq C_2\, g(x),\quad x\in I.
\end{equation*}
Sometimes we also write $C=C(\alpha)$ to emphasize, for a positive  constant $C>0$, its dependence on the parameter $\alpha$.

\bigskip
\bigskip
\noindent\textbf{Acknowledgment.}
The authors are members of the INDAM (``Istituto Nazionale di Alta Matematica'') research group GNAMPA (``Gruppo Nazionale per l’Analisi Matematica, la Probabilità e le loro Applicazioni'').

\section{The operator  in $L^p_{m}$ and upper kernel estimates}\label{section general operator}
We consider the singular elliptic operator
\begin{align}\label{general operator def}
	\mathcal L&=\mbox{Tr }\left(AD^2\right)+\frac{\left(v,\nabla\right)}y,
\end{align}
in the half-space $\R^{N+1}_+=\{(x,y): x \in \R^N, y>0\}$. Here   $A=\left(a_{ij}\right)\in \R^{N+1,N+1}$ is a symmetric and positive definite 
 matrix and we write $\mbox{Tr }\left(AD^2\right)=\sum_{i,j=1}^{N+1}a_{ij}D_{ij}$. The vector  $v=(d,c)\in\R^{N+1}$    satisfies 
\begin{align*}
	d=0\qquad \quad \text{if} \qquad\quad  c= 0
\end{align*}
i.e. $v$ is  oblique with respect to the boundary of $\R^{N+1}_+$.
 Writing
\begin{align*}
	A:=	\left(
	\begin{array}{c|c}
		Q  & { q}^t \\[1ex] \hline
		q& \gamma
	\end{array}\right)
\end{align*}
where $Q\in \R^{N\times N}$,  $q=(q_1, \dots, q_N)\in\R^N$ and $\gamma=a_{N+1, N+1}>0$, the operator  $\mathcal L$ takes the form
\begin{align}\label{def L transf}
		\mathcal L&=\mbox{Tr }\left(QD^2_x\right)+2\left(q,\nabla_xD_y\right)+\gamma D_{yy}+\frac{ d\cdot \nabla_x+cD_y }{y}.
\end{align}

\smallskip 
We work in the weighted spaces $L^p_m(\R_+^{N+1}):=L^p(\R_+^{N+1}; y^m dx dy)$, where $m\in\R$ and we  often write only $L^p_m$ when $\R^{N+1}_+$ is understood. In this section we recall the  results about generation of semigroups, maximal regularity and domain characterization in $L^p_m$  proved in \cite{MNS-Singular-Half-Space, Negro-AlphaDirichlet}.

We    endow $\mathcal L$ with   Neumann or  oblique derivative boundary conditions 
$$\lim_{y\to 0} D_y u=0\quad \text{(if $v=0$)},\qquad\qquad  \lim_{y\to 0}y^{\frac c\gamma}\, v \cdot \nabla u=0\quad \text{(if $c\neq 0$)}.$$

The latter  conditions can be equivalently written in   integral form 
(see \cite[Proposition 4.9]{Negro-AlphaDirichlet} and \cite[Proposition 4.6]{MNS-Sobolev}:  we define accordingly the  weighted Sobolev spaces
\begin{align*}
	W^{2,p}(m)=&\left\{u\in W^{2,p}_{loc}(\R^{N+1}_+):\    D^2u,\  \nabla u,\ u\in L^p_m\right\}
\end{align*}
and we then impose the boundary conditions thus defining
\begin{align}\label{Definition W_v alpha}
\nonumber 	
	W^{2,p}_{v}(m)&{=}\{u \in W^{2,p}(m):\ y^{-1}v\cdot \nabla u \in L^p_m\},\qquad (c\neq 0);\\[2ex]
W^{2,p}_{\mathcal N}(m)&{=}\{u \in W^{2,p}(m):\ y^{-1}D_yu \in L^p_m\},\qquad \hspace{1.8ex}(v= 0),
\end{align}
where, to keep consistency of the notation we write  when $d=0$,\, $W^{2,p}_{(0,c)}(m){=}W^{2,p}_{\mathcal N}(m)$.

\begin{teo}{{\em(\cite[Theorem 5.2]{Negro-AlphaDirichlet} and \cite[Theorem 7.7]{MNS-Singular-Half-Space})}}\label{Teo gen in Lpm}
	Let $v=(d,c)\in\R^{N+1}$ with $d=0$ if $c=0$, and let 
	$$0<\frac{m+1}p<\frac{c}{\gamma}+1.$$   
	Then the operator
	\begin{align*}
		\mathcal L&=\mbox{Tr }\left(AD^2\right)+\frac{\left(v,\nabla\right)}y
	\end{align*}
endowed with domain 	$D(\mathcal L)=W^{2,p}_{v}(m)$	generates a bounded analytic semigroup  in $L^p_m$ which has maximal regularity.
\end{teo}
\begin{cor}\label{Oblique cor1}{{\em(\cite[Corollary 5.3]{Negro-AlphaDirichlet})}}
	Under the assumptions of the previous theorem, the estimate
	\begin{align}\label{elliptic regularity oblique} 
		\| D^2 u\|_{L^p_{m}} +\|y^{-1} v\cdot \nabla u\|_{L^p_{m}}\leq C\| \mathcal Lu\|_{L^p_{m}}
	\end{align}
	\mbox{}\\[-1.5ex]	holds for every $u \in W^{2,p}_{v}(m)$ (if $c=0$ replace $y^{-1} v\cdot \nabla u$ with $y^{-1} D_yu$).
\end{cor}
The aim of the paper is to prove  that, if  $\frac c \gamma+1>0$,  then the  heat kernel $p_{{\mathcal L}}$ of $\mathcal L$, written  with respect the measure $y^\frac{c}{\gamma}dz$, satisfies the upper and lower Gaussian estimates
\begin{align}\label{up low section 1}
	 p_{{\mathcal  L}}(t,z_1,z_2)
	\simeq C t^{-\frac{N+1}{2}} y_1^{-\frac{c}{2\gamma}} \left(1\wedge \frac {y_1}{\sqrt t}\right)^{\frac{c}{2\gamma}} y_2^{-\frac{c}{2\gamma}} \left(1\wedge \frac{y_2}{\sqrt t}\right)^{\frac{c}{2\gamma}}\,\exp\left(-\dfrac{|z_1-z_2|^2}{kt}\right),
\end{align}
where $t>0$, $z_1=(x_1,y_1),\ z_2=(x_2,y_2)\in\R^{N+1}_+$ and $C,k$ are some positive constants which may differ from the upper and the lower bounds. 
\begin{os}\label{oss equiv estimate}
	We emphasize that the above estimate can be written  equivalently  as
	\begin{align}\label{upper estimates ver2}
		p_{{\mathcal  L}}(t,z_1,z_2)
		\simeq C t^{-\frac{N+1}{2}}  y_i^{-c} \left(1\wedge \frac{y_i}{\sqrt t}\right)^{c}\,\exp\left(-\dfrac{|z_1-z_2|^2}{kt}\right),\qquad i=1,2
	\end{align}
	for some possibly different constants $C,k>0$. This is a consequence of \cite[Lemma 10.2]{MNS-Caffarelli} which guarantees that for any $\epsilon>0$ one has 
	\begin{align}\label{equiv estimates up low}
		y_1^{-\frac{c}2} \left(1\wedge y_1\right)^{\frac{c}2}\simeq 	y_2^{-\frac{c}2} \left(1\wedge y_2\right)^{\frac{c}2}\,\exp\left(\epsilon|y_1-y_2|^2\right),\qquad \forall y_1,y_2>0
	\end{align}
\end{os}
\medskip

In \cite{Negro-Spina-SingularKernel} we already  proved the  upper kernel estimate of \eqref{up low section 1}: we recall here the main results referring also to Section \ref{L2} for the construction of $\mathcal L$ in the space $L^2_{\frac c\gamma}$ via form methods.

\begin{teo}{{\em(\cite[Theorems 4.18 and 5.3]{Negro-Spina-SingularKernel})}}\label{true.kernel} 
Let $\frac{c}{\gamma}+1>0$.  Then for some  $\theta\in (0,\frac \pi 2)$, the operator  $\mathcal L$	generates  in $L^2_{\frac c\gamma}$ a contractive analytic semigroup  $(e^{t{\mathcal L}})_{t\in \Sigma_{\frac{\pi}{2}-\theta}}$ which is positive for $t>0$. The semigroup   consists of integral operators i.e. there exists $p_{{\mathcal L}}(t,\cdot,\cdot)\in L^\infty(\R^{N+1}_+\times\R^{N+1}_+)$ 
such that  for $t\in\Sigma_{\frac{\pi}{2}-\theta}$, $z_1=(x_1,y_1),\ z_2=(x_2,y_2)\in\R^{N+1}_+$
\begin{align*}
	e^{t{\mathcal  L}}f(z_1)=
	\int_{\R^{N+1}_+}p_{{\mathcal L}}(t,z_1,z_2)f(z_2)\,y_2^{\frac c\gamma}\,dz_2,\quad f\in L^2_{{\frac c\gamma}}.
\end{align*}
Moreover for every $\eps>0$  there exist  constants $C_{\eps}, k_\eps>0$ such that if  $t\in\Sigma_{\frac{\pi}{2}-\theta-\eps}$ then
\begin{align*}
	|p_{{\mathcal  L}}(t,z_1,z_2)|
	\leq C_\epsilon |t|^{-\frac{N+1}{2}} y_1^{-\frac{c}{2\gamma}} \left(1\wedge \frac {y_1}{\sqrt {|t|}}\right)^{\frac{c}{2\gamma}} y_2^{-\frac{c}{2\gamma}} \left(1\wedge \frac{y_2}{\sqrt {|t|}}\right)^{\frac{c}{2\gamma}}\,\exp\left(-\dfrac{|z_1-z_2|^2}{k_\epsilon {|t|}}\right).
\end{align*}
\end{teo} 
\bigskip

In view of the previous Theorem   the remaining part of the paper is then  devoted to   prove the  lower bound of \eqref{up low section 1}.

\section{Similarity transformations and the model operator}\label{Section oblique}
In this section we show that some simple change of variables allows to  work, in what follows  and without any loss of generality,  with the model operator
\begin{equation}\label{model operator}
	\Delta_{x} +2a\cdot \nabla_xD_yu+D_{yy}+\frac{c}{y}D_y:=\Delta_x u+2a\cdot \nabla_xD_yu+ B_yu 
\end{equation} under Neumann boundary conditions. Here $B_y=D_{yy}+\frac{c}{y}D_y$ is a Bessel operator and  $a\in\R^N$ satisfies the requirement $|a| <1$ which is equivalent to the ellipticity of the top order coefficients. We briefly describe the strategy referring to \cite[Section 8]{MNS-Singular-Half-Space}  and \cite[Section 5]{Negro-Spina-SingularKernel} for further details. Let 
\begin{align*}
	\mathcal L&=\mbox{Tr }\left(AD^2\right)+\frac{\left(v,\nabla\right)}y=\mbox{Tr }\left(QD^2_x\right)+2\left(q,\nabla_xD_y\right)+\gamma D_{yy}+\frac{ d\cdot \nabla_x+cD_y }{y}
\end{align*}
be the operator in \eqref{general operator def}.
If $d\neq 0$ and $c\neq 0$, a  first transformation allows to transform $\mathcal L$ into a similar operator with $d=0$ and Neumann boundary conditions. This is a consequence of the following lemma which follows by straightforward computation. Let us consider the following isometry of $L^p_{m}$
\begin{align}\label{Tran def}
	T\, u(x,y)&:=u\left(x-\frac d c y,y\right),\quad (x,y)\in\R_+^{N+1}.
\end{align}

\begin{lem}\label{Isometry action der} Let $1< p< \infty$,  $ v=(d,c)\in\R^{N+1}$,  $c\neq 0$. Then  for $u\in W^{2,1}_{loc}\left(\R_+^{N+1}\right)$ 
		\begin{itemize}
			\item[(i)] $\ds T^{-1}\,\left(v\cdot \nabla \right) T\,u=cD_yu$;
			
			\item[(ii)] $\ds 	T^{-1}\,\left(\mbox{Tr }\left(AD^2\right)+\frac{ v\cdot \nabla  }{y}\right) T\,u=\mbox{Tr }\left(\tilde AD^2u\right)+\frac{c}{y}D_y u$.
		\end{itemize}
		Here $\tilde A$  is a uniformly elliptic symmetric matrix defined by 
		\begin{align*}
			\tilde A=
			\left(
			\begin{array}{c|c}
				Q-\frac{2}c d\otimes q{+\frac \gamma{c^2} d\otimes d}  & q^t- \frac{\gamma}c d^t \\[1ex] \hline
				q- \frac{\gamma}c d & \gamma
			\end{array}\right)
		\end{align*} 
		and  $\gamma= q_{N+1,N+1}$.
	\end{lem}{\sc{Proof.}} The proof follows by a straightforward computation.
	\qed
	\medskip

	The latter lemma shows that the operator $\mathcal L$, endowed with the oblique  boundary condition $\lim_{y \to 0} y^{\frac c \gamma}\, v \cdot \nabla u=0$ is unitarily equivalent to  the operator
	\begin{align*}
		\mathcal {\tilde{L}}:=T^{-1}\mathcal L T\,=\mbox{Tr }\left(\tilde AD^2\right)+\frac{c}{y}D_y=\mbox{Tr }\left(\tilde QD^2_x\right)+2\left(\tilde q,\nabla_xD_y\right)+\gamma D_{yy}+\frac{c}{y}D_y  
	\end{align*}
	with the Neumann  boundary condition $\lim_{y \to 0} y^{\frac c \gamma}\, D_yu=0$.  Then through a linear change of variables in the $x$ variables, the term $\mbox{Tr }\left(\tilde QD^2_x\right)$ can be finally transformed into $\gamma\Delta_x$ thus obtaining an operator which is similar to the one in \eqref{model operator}.\\

In view of the above simplification, in the next sections, without any loss of generality, we will work almost exclusively with the model operator  \eqref{model operator}, referring to Theorem \ref{complete-sharp} for the generalization of our results to the general case. \\\

We start in the next section with the $L^2$-theory. In order to control the non self-adjointness of $\mathcal L$, we need to simultaneously study both $\mathcal L$ and $\mathcal L^\ast$.

\section{The operator $\mathcal L=\Delta_{x} +2a\cdot\nabla_xD_y+ B_yu$ in $L^2_{c}$} \label{L2}
Let $c\in\R$, $a=(a_1, \dots, a_N) \in\R^N$ such that $c+1>0$, $|a|<1$. In this section we study the $L^2$ theory related to  the  model degenerate   operator 
\begin{equation}  \label{La}
	\mathcal L :=\Delta_x u+2a\cdot \nabla_xD_yu+ B_yu 
\end{equation}
in $L^2_c$ equipped with Neumann boundary condition.  The requirement $|a| <1$ is equivalent to the ellipticity of the top order coefficients.

We use the Sobolev space $H^{1}_{c}:=\{u \in L^2_{c} : \nabla u \in L^2_{c}\}$ equipped with the inner product
\begin{align*}
	\left\langle u, v\right\rangle_{H^1_{c}}:= \left\langle u, v\right\rangle_{L^2_{c}}+\left\langle \nabla u, \nabla v\right\rangle_{L^2_c}.
\end{align*}

The condition $c+1>0$ assures, by \cite[Theorem 4.9]{MNS-Sobolev}, that the set 
\begin{equation} \label{defC}
	\mathcal{C}:=\left \{u \in C_c^\infty \left(\R^N\times[0, \infty)\right), \ D_y u(x,y)=0\  {\rm for} \ y \leq \delta\ {\rm  and \ some\ } \delta>0\right \},
\end{equation}
is dense in $H^1_c$. Moreover,  by \cite[Remark 4.14]{MNS-Sobolev}, if a function $u\in H^1_c$ has  support in $\R^N\times[0,b]$, then there exists a sequence $\left(u_n\right)_{n\in\N}\in\mathcal C$  such that $ \mbox{supp }u_n\subseteq \R^N\times[0,b]$  and  $u_n\to u$ in $H^1_c$.

We consider the  form  in $L^2_{c}$ 
\begin{align}\label{form a def}
	\nonumber \mathfrak{a}(u,v)
	&:=
	\int_{\R^{N+1}_+} \langle \nabla u, \nabla \overline{v}\rangle\,y^{c} dx\,dy+2\int_{\R^{N+1}_+} D_yu\, a\cdot\nabla_x\overline{v}\,y^{c} dx\,dy,\\[1ex]
	D(\mathfrak{a})&:=H^1_{c}
\end{align}
and its adjoint $\mathfrak{a}^*(u,v)=\overline{\mathfrak{a}(v,u)}$  
\begin{align*}
	\mathfrak{a^*}(u,v)=\overline{\mathfrak{a}(v,u)}
	&:=
	\int_{\R^{N+1}_+} \langle \nabla u, \nabla \overline{v}\rangle\,y^{c} dx\,dy+2\int_{\R^{N+1}_+} a\cdot \nabla_x u\,D_y\overline{v}\,y^{c} dx\,dy.
\end{align*}

\begin{prop}{{\em (\cite[Proposition 2.1]{Negro-Spina-SingularKernel})}} \label{prop-form}
	The forms $\mathfrak{a}$, $\mathfrak{a^*}$ are continuous, accretive and   sectorial.
\end{prop}
%
%
%
We define the operators $\mathcal L$ and $\mathcal L^*$ associated respectively to the forms $\mathfrak{a}$ and $\mathfrak{a}^*$  by
\begin{align} \label{BesselN}
	\nonumber D( \mathcal L)&=\{u \in H^1_{c}: \exists  f \in L^2_{c} \ {\rm such\ that}\  \mathfrak{a}(u,v)=\int_{\R^{N+1}_+} f \overline{v}y^{c}\, dz\ {\rm for\ every}\ v\in H^1_{c}\},\\  \mathcal Lu&=-f;
\end{align}
\begin{align} \label{adjoint}
	\nonumber D( \mathcal L^*)&=\{u \in H^1_{c}: \exists  f \in L^2_{c} \ {\rm such\ that}\  \mathfrak{a}^*(u,v)=\int_{\R^{N+1}_+} f \overline{v}y^{c}\, dz\ {\rm for\ every}\ v\in H^1_{c}\},\\  \mathcal L^*u&=-f.
\end{align}
If $u,v$ are smooth functions with compact support in the closure of $\R_+^{N+1}$ (so that they do not need to vanish on the boundary), it is easy to see integrating by parts that $$-\mathfrak a (u,v)= \langle \Delta_x u+2a\cdot \nabla_xD_yu+ B_yu, \overline v\rangle_{L^2_c}  $$
if $\ds \lim_{y \to 0} y^c D_yu=0$. This means that 
$\mathcal L$ is the operator 
$$\Delta_x +2a\cdot \nabla_xD_y+ B_y$$
 with Neumann boundary conditions at $y=0$. On the other hand 
$$-\mathfrak a^* (u,v)= \left\langle \Delta_x u+2a\cdot \nabla_xD_yu+2c\frac {a\cdot \nabla_x u}{y}+ B_yu, \overline v\right\rangle_{L^2_c}  $$
if $\ds \lim_{y \to 0} y^c \left ( 2a\cdot \nabla_x u+D_yu\right )=0$ and therefore the adjoint  of $\mathcal L$ is the operator
\begin{align*}
	\mathcal L^*=\Delta_x +2a\cdot \nabla_xD_y+2c\frac {a\cdot \nabla_x u}{y}+ B_y
\end{align*} with the above oblique condition at $y=0$.  We refer the reader to \cite{MNS-Sobolev} for further details about the Sobolev spaces $H^1_c$ and their boundary conditions at $y=0$.
\begin{prop}{{\em (\cite[Proposition 2.2]{Negro-Spina-SingularKernel})}}\label{generation L2}
	$\mathcal L$ and $\mathcal L^*$  generate contractive analytic semigroups $e^{t \mathcal L}$, $e^{t \mathcal L^*}$, $t\in\Sigma_{\frac{\pi}{2}-\arctan \frac{|a|}{1-|a|}}$,  in $L^2_{c}$. Moreover the semigroups $(e^{t\mathcal L})_{t \geq 0}, (e^{t\mathcal L^*})_{t \geq 0}$ are positive and $L^p_c$-contractive for $1 \leq p \leq \infty$.
\end{prop}
	%
%

The previous proposition is a special case of  Theorem \ref{Teo gen in Lpm} which  shows that the  semigroups $e^{t \mathcal L}$, $e^{t \mathcal L^*}$ extrapolate to the weighted spaces $L^p\left(\R^{N+1}_+, y^mdxdy\right)$, $m\in\R$, when   $0<\frac{m+1}p<c+1$, and which describes  the domain of their generators as respectively the weighted Sobolev spaces defined in \eqref{Definition W_v alpha}
\begin{align*}
	W^{2,p}_{\mathcal N}(m),\qquad \qquad 	W^{2,p}_{v}(m),\quad v=(2a,1).
\end{align*}

\medskip

Finally we collect in the following proposition some invariance properties of $\mathcal L$, $e^{t\mathcal L}$ and of of $\mathcal L^\ast$, $e^{t\mathcal L^\ast}$ whose proofs follow after an easy computation.

\begin{prop}\label{Prop scaling} The following properties hold.
	\begin{itemize}
		\item[(i)] The scale homogeneity of $\mathcal L$ is  $2$:
		\begin{align}\label{scale.op}
			s^2\mathcal L=I_s^{-1}\mathcal L I_s, \qquad I_su(x,y)=u(sx,sy), \qquad s>0. 
		\end{align}
\item[(ii)]   $\mathcal L$ is invariant under translation in the $x$-direction:
\begin{align}\label{scale.op}
	\mathcal L=T_{x_0}^{-1}\mathcal L T_{x_0}, \qquad T_{x_0}u(x,y)=u(x+x_0,y), \qquad x_0\in\R^N. 
\end{align}
	\end{itemize}	Moreover the semigroup $e^{t\mathcal L}$ generated by $\mathcal L$  in $L^2_{c}$
		satisfies for any $t\in\Sigma_{\frac{\pi}{2}-\arctan \frac{|a|}{1-|a|}}$
		\begin{equation}\label{scale.semi}
			e^{s^2 t\mathcal L}=I_{s}^{-1}e^{t\mathcal L}I_s,\quad s>0,\qquad e^{ t\mathcal L}=T_{x_0}^{-1}e^{t\mathcal L}T_{x_0},\quad x_0\in\R^N.  
		\end{equation}
The same results  hold for the adjoint operator $\mathcal L^\ast$ and its generated semigroup. 
\end{prop}
\medskip
By Theorem \ref{true.kernel} , the semigroup $e^{t{\mathcal L}}$,  $t\in\Sigma_{\frac{\pi}{2}-\arctan \frac{|a|}{1-|a|}}$,   consists of integral operators and its heat kernel  $p_{{\mathcal L}}$, written with respect to the measure $y^cdz$, satisfies the upper estimate
\begin{align}\label{upper estimates La}
	0\leq p_{{\mathcal  L}}(t,z_1,z_2)
	\leq C t^{-\frac{N+1}{2}} y_1^{-\frac{c}{2}} \left(1\wedge \frac {y_1}{\sqrt t}\right)^{\frac{c}{2}} y_2^{-\frac{c}{2}} \left(1\wedge \frac{y_2}{\sqrt t}\right)^{\frac{c}{2}}\,\exp\left(-\dfrac{|z_1-z_2|^2}{k t}\right).
\end{align}
which is valid for any $t>0$, $z_1,z_2\in\R^{N+1}_+$ and for some positive constant $C,k>0$. By remark \ref{oss equiv estimate} the above estimate can be written  equivalently  as
\begin{align}\label{upper estimates ver2 La}
	p_{{\mathcal  L}}(t,z_1,z_2)
	\leq  C t^{-\frac{N+1}{2}}  y_i^{-c} \left(1\wedge \frac{y_i}{\sqrt t}\right)^{c}\,\exp\left(-\dfrac{|z_1-z_2|^2}{kt}\right),\qquad i=1,2
\end{align} 
The latter bounds extends also, for some possibly different $C,k>0$, for $t$ in  any sub-sector $\Sigma_{\frac{\pi}{2}-\arctan \frac{|a|}{1-|a|}-\eps}$.
Moreover,  using  the invariance properties of Proposition \ref{Prop scaling}, we have that  for any $t\in\Sigma_{\frac\pi 2-\arctan\frac{|a|}{1-|a|}}$, $z_1,\ z_2\in\R^{N+1}_+$, 
\begin{align}\label{scaling}
	p_{\mathcal{L}}(s^2t,z_1,z_2)=s^{-(N+1+c)}\,p_{\mathcal{L}}(t,s^{-1}z_1,s^{-1}{z_2}),\qquad s>0,
\end{align}
and
\begin{align}\label{invariance x-translation}
	p_{\mathcal{L}}(t,z_1+x_0,z_2+x_0)=p_{\mathcal{L}}(t,z_1,z_2),\qquad x_0\in\R^N,
\end{align}
where for $z=(x,y)\in\R^{N+1}_+$, we wrote,  with a little abuse of notation,  $z+x_0=\left(x+x_0,y\right)$ to denote the translation in the $x_0$-direction. \\

The same results  hold for the adjoint operator $\mathcal L^\ast$ and its generated semigroup. In particular, if  $p_{{\mathcal L^\ast}}$ is the heat kernel of $\mathcal L^\ast$ , written with respect to the measure $y^cdz$, then by definition and by the positivity of the semigroups, the following relations hold for any $z_1,\ z_2\in\R^{N+1}_+$:
\begin{align}
		\label{adjoint-ker}
\nonumber 	p_{{\mathcal L^\ast}}(t,z_1,z_2)&=\overline{p_{{\mathcal L}}(t,z_2,z_1)},\qquad\qquad  \forall t\in\Sigma_{\frac\pi 2-\arctan\frac{|a|}{1-|a|}},\\[1ex]
	p_{{\mathcal L^\ast}}(t,z_1,z_2)&=p_{{\mathcal L}}(t,z_2,z_1),\qquad\qquad  \forall t>0.
\end{align}

\section{Gradient estimates}\label{section gradient estimate}
In this section we prove gradient estimates for the heat kernels of $\mathcal L$ and $\mathcal L^\ast$. For simplicity in what follows we only work with $\mathcal L$ but all the  results hold with similar proof also for $\mathcal L^\ast$. 

We start by stating  the next lemma in which we specialize Theorem \ref{Teo gen in Lpm} and Corollary \ref{Oblique cor1} to  the case of  $L^p:=L^p\left(\R^{N+1}_+\right)$ with Lebesgue measure. In what follows, recalling  \eqref{Definition W_v alpha}, we write 
\begin{align*}
	W^{2,p}_{\mathcal{N}}:=W^{2,p}_{\mathcal{N}}(0)=\{u \in W^{2,p}\left(\R^{N+1}_+\right):\ y^{-1}D_yu \in L^p\}.
\end{align*}
Note that from Morrey's embedding, if $p>N+1$, then $W^{2,p}_{\mathcal{N}}\hookrightarrow C^1_b\left(\overline{\R^{N+1}_+}\right)$.
\begin{lem}\label{lem stime Lb inf}
	Let $0<\frac{1}p<c+1$. 	Then $\mathcal L$ endowed with domain 	$D\left(\mathcal L\right)=W^{2,p}_{\mathcal{N}}$	generates a bounded analytic semigroup  in $L^p$. Moreover there exists $C>0$ such that for $\lambda>0$ and  $u \in W^{2,p}_{\mathcal{N}}$ one has 
	$$
	\lambda\|u\|_p+\lambda^\frac{1}{2}\|\nabla u\|_p+\|D^2 u\|_p \leq C\|\lambda u- \mathcal Lu\|_{p}.
	$$
\end{lem} 
{\sc Proof.} Recalling  Theorem \ref{Teo gen in Lpm} and Corollary \ref{Oblique cor1},   $(\mathcal L,W^{2,p}_{\mathcal{N}}) $ generates a bounded semigroup on $L^p\left(\R^{N+1}_+\right)$ and moreover we have for $u\in W^{2,p}_{\mathcal{N}}$, $\lambda >0$
\begin{align*}
	\lambda\|u\|_p\leq C\left( \|\lambda u-\mathcal L u\|_p\right), \quad \|D^2u\|_p \leq C\|\mathcal Lu\|_p\leq C(\lambda \|u\|_p+\|\lambda u-\mathcal Lu\|_p).
\end{align*}
The estimate of the gradient term follows by the interpolative inequality $\|\nabla u\|_p^2 \leq C\|u\|_p \|D^2u\|_p$.
\qed

The following regularity result  is an immediate consequence of the holomorphy of the semigroup $e^{t \mathcal L}$.

\begin{lem} \label{regpb}
For every fixed $z_1,\ z_2\in\R^{N+1}_+$ the kernel $p_\mathcal L(t, z_1, z_2)$ is holomorphic with respect to $t\in\Sigma_{\frac\pi 2-\arctan\frac{|a|}{1-|a|}}$  and  $p(t, \cdot , z_2) \in C^1_b\left(\overline{\R^{N+1}_+}\right)$.
\end{lem}
{\sc Proof.}  Let us fix  $p>\max\{(c+1)^{-1}, N+1\}$ and let us work, using Lemma \ref{lem stime Lb inf},  in $L^p:=L^p\left(\R^{N+1}_+\right)$. If $s>0$, by \eqref{upper estimates ver2 La}, $p_\mathcal L(s, \cdot, z_2) \in L^p$ and then,  since the semigroup is analytic, $e^{t \mathcal L} p_\mathcal L$ belongs to the domain of $\mathcal L$,  $D(\mathcal L)=W^{2,p}_{\mathcal{N}}$. Since $e^{t \mathcal L} p_\mathcal L(s,z_1,z_2)=p_\mathcal L(t+s,z_1,z_2)$, by the semigroup law, we have that $p(t+s, \cdot, z_2) \in W^{2,p}_{\mathcal{N}}\hookrightarrow C^1_b\left(\overline{\R^{N+1}_+}\right)$. The analyticity with respect to $t\in\Sigma_{\frac\pi 2-\arctan\frac{|a|}{1-|a|}}$  follows again by the identity $e^{t \mathcal L} p_\mathcal L(s,z_1,z_2)=p_\mathcal L(t+s,z_1,z_2)$, using the analyticity of the semigroup.
\qed

\smallskip 
The Cauchy formula for the derivatives of holomorphic functions allows to estimate $D_tp_\mathcal L$ and  $\mathcal L p_\mathcal L$.
	\begin{prop}\label{Time derivative estimates}
			Let $c+1>0$, $\eps>0$. Then   there exists $C,\ k>0$, such that, for every  $t\in\Sigma_{\frac\pi 2-\arctan\frac{|a|}{1-|a|}-\eps}$ and almost every $z_1,\ z_2\in\R^{N+1}_+$,
		\begin{align*} 
			\left|\mathcal L p_\mathcal L(t,z_1,z_2)\right|+	\left|D_tp_\mathcal L (t,z_1,z_2)\right|
			\	\leq C |t|^{-\frac{N+3}{2}}  y_2^{-c} \left(1\wedge \frac{y_2}{\sqrt{ |t|}}\right)^{c}\,\exp\left(-\dfrac{|z_1-z_2|^2}{k|t|}\right).
 \end{align*}
	\end{prop}
{\sc{Proof.}}
		Since the kernel $p_\mathcal L$  satisfies the equation $D_tp_\mathcal L=\mathcal L p_\mathcal L$,   it is sufficient to deal only with $D_t p_\mathcal L$.   Let  $r$ be small enough such  that
		\begin{align*}
			B\left(t_0,  r|t_0|\right)\subset \Sigma_{\frac\pi 2-\arctan\frac{|a|}{1-|a|}}, \qquad \forall t_0\in\Sigma_{\frac\pi 2-\arctan\frac{|a|}{1-|a|}-\eps}.
		\end{align*}
			Using the Cauchy formula for the derivatives of holomorphic functions in the ball  $B\left(t_0,  r|t_0|\right)$, we get 
		\begin{align*} 
			\left |D_t p_\mathcal L(t_0,y,\rho)\right|
			\leq \frac{1}{r|t_0|}\max_{|t-t_0|=r|t_0|}|p_\mathcal L(t,z_1,z_2)|,\qquad z_1,\ z_2\in\R^{N+1}_+.
		\end{align*}
		Applying the estimate of  Theorem \ref{true.kernel} in the sector $\Sigma_{\frac\pi 2-\arctan\frac{|a|}{1-|a|}}$ and recalling Remark \ref{oss equiv estimate}, we obtain for suitable $C', \kappa'$
		\begin{align*} 
		\left |D_t p_\mathcal L(t_0,y,\rho)\right|\leq C'|t|^{-\frac{N+3}{2}}  
			   y_2^{-c}\left (\frac{y_2}{|t_0|^{\frac{1}{2}}}\wedge 1 \right)^{c}
			\exp\left(-\frac{|z_1-z_2|^2}{\kappa' |z_0|}\right)
		\end{align*}
	which is equivalent to the statement. 
\qed

Before proving the estimates for  the gradient of  $p_\mathcal L$,  we recall that,  using  the scaling equalities of Proposition \ref{Prop scaling}, we have that  for any $t\in\Sigma_{\frac\pi 2-\arctan\frac{|a|}{1-|a|}}$, ${z_1}, {z_2}\in \R^{N+1}_+$ 
\begin{align} \label{scaling 1}
\nonumber	p_{\mathcal{L}}(t,z_1,z_2)&=|t|^{-\frac{N+1+c}2}p_{\mathcal{L}}\left(t|t|^{-1},\frac{z_1}{\sqrt {|t|}},\frac{z_2}{\sqrt {|t|}}\right),\\[1ex] 
	\nabla_{z_1}p_{\mathcal{L}}(t,z_1,z_2)&=|t|^{-\frac{N+2+c}2}\nabla p_{\mathcal{L}}\left(t|t|^{-1},\frac{z_1}{\sqrt {|t|}},\frac{z_2}{\sqrt {|t|}}\right).
\end{align}

Now we localize the  estimates of Lemma \ref{lem stime Lb inf}. For   $z\in\R^{N+1}_+$, $r>0$ we set $B^+(z,r):=B(z,r)\cap\R^{N+1}_+$. 
	\begin{prop}\label{Proposition inner estimates} Let $c+1>0$, $p>\max\{(c+1)^{-1}, N+1\}$. Then there exists a constant $C>0$  such that  for every $u \in W^{2,p}_{\mathcal{N}}$, $r>0$,
	\begin{align*}
		\|\nabla u \|_{L^\infty(B_{\frac r 2}^+)}\leq Cr^{-\frac{N+1}{p}}\left(r^2\|\mathcal L u\|_{L^p(B_r^+)}+\|u\|_{L^p(B_r^+)}\right)
	\end{align*}
\end{prop}
{\sc Proof.}  Up to a scaling argument we can suppose, without any loss of generality, $r=1$. Set $r_n=\sum_{k=1}^n2^{-k}$. Then $r_1=1/2$, $r_\infty=1$, $r_{n+1}-r_n=2^{-(n+1)}$.

Let  $B^+_n=B^+(z,r_n)$, $B^+= B^+(z,1)$, $B_{\frac 1 2}^+= B^+(z,1/2)$  and choose  cut-off functions $\eta_n\in C_c^{\infty }(\R^{N+1})$ such that  $\eta_n(x,y)=\eta_n(x,-y)$,
$0\le \eta_n\le 1$, 
$\eta_n=1$ in $B_n^+$, $({\rm supp \ } \eta_n) \cap \R^{N+1}_+ \subset B_{n+1}^+$,
$|\nabla \eta_n| \leq \frac{C}{r}2^n$,  $|D^2 \eta_n| \leq \frac{C}{r^2}4^n$ for some constant $C>0$ independent of $n$. Then also $|y^{-1}D_y \eta_n| \leq \frac{C}{r^2}4^n$, since $D_y \eta_n(x,0)=0$.  If  $u\in W^{2,p}_{\mathcal{N}}$ then  $\eta_n u\in W^{2,p}_{\mathcal{N}}$   and we have
\begin{align*}
	\mathcal L(\eta_n u)=&\eta_n \mathcal L u +2\nabla_x\eta_n\nabla_xu+2 D_y \eta_n D_y u+2a\cdot(\nabla_x\eta_nD_yu+D_y\eta_n\nabla_xu)\\& +u\left(D_{yy}\eta_n+\Delta_x\eta_n+2a\cdot\nabla_xD_y\eta_n+c\frac{D_y\eta_n}{y}\right).
\end{align*}
Applying Lemma \ref{lem stime Lb inf} with $\lambda=1$ to $\eta_n u$   we get 
\begin{align*}
	&\|\eta_nu\|_p+\|\nabla (\eta_nu)\|_p+\|D^2 (\eta_nu)\|_p
	\leq C \left(\|\mathcal L (\eta_nu)\|_p+\|\eta_nu\|_p\right)\\[1.5ex]
	&\leq C\left( \|\mathcal L u\|_{L^p(B^+)}+2^n\|\nabla(\eta_{n+1}u)\|_{p}+\left(4^n+1\right)\|u\|_{L^p(B_r^+)}\right).
\end{align*}
Applying the interpolative inequality $\|\nabla u\|_p \leq C\epsilon  \|D^2u\|_p^2+\frac 1 \epsilon \|u\|_p$ we get for $\epsilon>0$
\begin{align*}
	&\|\eta_nu\|_p+\|\nabla (\eta_nu)\|_p+\|D^2 (\eta_nu)\|_p \\[1.5ex]
	&\leq C\left(\|\mathcal L u\|_{L^p(B^+)}+\epsilon 2^n\|D^2(\eta_{n+1}u)\|_{p}+\frac{2^n}{\epsilon }\|\eta_{n+1}u\|_{p}+\left(4^n+1\right)\|u\|_{L^p(B^+)}\right)\\[1.5ex]
	&\leq C\left(\|\mathcal L u\|_{L^p(B^+)}+\epsilon 2^n\|D^2(\eta_{n+1}u)\|_{p}+\left(\frac{2^n}{\epsilon }+4^n+1\right)\|u\|_{L^p(B_r^+)}\right)
\end{align*}
Setting
$\xi:=C2^n\eps $, we get
\begin{align*}
	& \|\eta_nu\|_p+\|\nabla (\eta_nu)\|_p+\|D^2 (\eta_nu)\|_p \\[1.5ex]
	&\leq C\left( \|  u-\mathcal L u\|_{L^p(B^+)}+\left(\frac{4^n}{\xi }+{4^n}\right)\|u\|_{L^p(B^+)}\right)+\xi \|D^2(\eta_{n+1}u)\|_{p}.
\end{align*}
It follows that 
\begin{align*}
	&\xi^n\Big( \|\eta_nu\|_p+\|\nabla (\eta_nu)\|_p+\|D^2(\eta_nu)\|_p \Big)\\[1.5ex]
	&\leq C\left( \xi^n\|  u-\mathcal L u\|_{L^p(B_r^+)}+\xi^n\left(\frac{4^n}{\xi }+{4^n}\right)\|u\|_{L^p(B^+)}\right)+\xi^{n+1} \|D^2(\eta_{n+1} u)\|_{p}.
\end{align*}
By choosing $\eps=\eps_n$  so that $\xi=\frac 18$ and summing up the  previous inequality over $n\in\N$ we get
\begin{align*}
	  \|\nabla u \|_{L^p(B_{\frac 1 2}^+)}+\sum_{n=1}^\infty \xi^n \|D^2 (\eta_nu)\|_p
	\leq &C\left(\|u-\mathcal Lu\|_{L^p(B^+)}+\|u\|_{L^p(B^+)}\right)\\[1.5ex]
	&+\sum_{n=1}^\infty \xi^{n+1} \|D^2 (\eta_{n+1}u)\|_p.
\end{align*}
Cancelling equal terms on both sides it follows that
\begin{align*}
	\|\nabla u \|_{L^p(B_{\frac 1 2}^+)}&+\frac{1}{8}\|D^2 u\|_{L^p(B_{\frac 1 2}^+)}\leq C\left(\| u-\mathcal L u\|_{L^p(B^+)}+\|u\|_{L^p(B^+)}\right).
\end{align*}
 Since $p>N+1$, using the  Morrey's embedding, the previous inequality implies
we get 
\begin{align*}
	\|\nabla u \|_{L^\infty(B_{\frac 1 2}^+)}&\leq C\left(\|\nabla u \|_{L^p(B_{\frac 1 2}^+)}+\|D^2 u\|_{L^p(B_{\frac 1 2}^+)}\right)\leq C\left(\| u-\mathcal L u\|_{L^p(B^+)}+\|u\|_{L^p(B^+)}\right)
\end{align*}
which implies the required claim for $r=1$.
\qed

\begin{os}\label{os grad ajoint}
		 We remark that the same results as those in  Lemmas \ref{lem stime Lb inf} and  \ref{regpb} and  Propositions \ref{Time derivative estimates}  and  \ref{Proposition inner estimates}  hold also for the adjoint  operator $\mathcal L^\ast $ under the oblique boundary  condition $\ds \lim_{y \to 0} y^c \left ( 2a\cdot \nabla_x u+D_yu\right )=0$.  To be concise  we do not state them explicitly but their proof follows by arguing as before and by using the generation and the  domain characterization for the operator $\mathcal L^*$ endowed with domain 	$D\left(\mathcal L\right)=W^{2,p}_{v}$, $v=(2a,1)$ (see Theorem \ref{Teo gen in Lpm}  and Corollary \ref{Oblique cor1}). 
\end{os}
We need the following basic estimate.
\begin{lem}\label{lemma loc gauss}
	Let $z_0, z_2 \in\R^{N+1}_+$. Then 
	\begin{align*}
\sup_{z_1\in B(z_0,\sqrt t) }\exp\left(-\frac{|z_1-z_2|^2}t\right)\leq e^{16}\exp\left(-\frac 9{16}\frac{|z_0-z_2|^2}t\right)
	\end{align*}
\end{lem}
{\sc{Proof.}}  If $|z_0-z_2|\leq 4\sqrt t$ then for every $z_1\in\R^{N+1}_+$
\begin{align*}
\exp\left(-\frac{|z_1-z_2|^2}t\right)\leq  e^{16}\exp\left(-\frac{|z_0-z_2|^2}t\right).
\end{align*}
If $|z_0-z_2|> 4\sqrt t$ and $z_1\in B(z_0,\sqrt t)$, then $|z_1-z_2|\geq |z_0-z_2|-|z_1-z_0| \geq |z_0-z_2|-\frac 1 4 |z_0-z_2|=\frac 3 4 |z_0-z_2|$ and 
\begin{align*}
	\exp\left(-\frac{|z_1-z_2|^2}t\right)\leq  \exp\left(-\frac 9{16}\frac{|z_0-z_2|^2}t\right).
\end{align*}
\qed 

We can finally  deduce  pointwise estimates for  the gradient of the kernel $p_{\mathcal L}(t,z_1,z_2)$. Recalling \eqref{adjoint-ker}, due  to the asymmetry of $p_{\mathcal L}$ and  in order to control both derivatives with respect to  $z_1,\ z_2\in\R^{N+1}_+$, we state them also for the adjoint operator $\mathcal L^\ast$. 
\begin{teo}\label{space derivative estimates}
	Let $c+1>0$, $\eps>0$.   Then  for  every $\epsilon>0$,  there exist $C,k>0$, such that, for every  $t\in\Sigma_{\frac\pi 2-\arctan\frac{|a|}{1-|a|}- \epsilon }$ and almost every $z_1,\ z_2\in\R^{N+1}_+$,
		\begin{align*} 
			\left|\nabla_{z_1} p_\mathcal L(t,z_1,z_2)\right|+	\left|\nabla_{z_1} p_{\mathcal L^\ast }(t,z_1,z_2)\right|
			\	\leq C |t|^{-\frac{N+2}{2}}  y_2^{-c} \left(1\wedge \frac{y_2}{\sqrt{ |t|}}\right)^{c}\,\exp\left(-\dfrac{|z_1-z_2|^2}{k|t|}\right).
 \end{align*}
\end{teo}
{\sc{Proof.}} Let us firstly prove the required estimate for $\nabla_{z_1} p_\mathcal L$. In virtue of the validity of the scaling property \eqref{scaling 1}, we may  assume that $|t|=1$.  Let $z_0,\, z_2\in\R^{N+1}_+$. Then applying Proposition \ref{Proposition inner estimates} with $r=1$  to the function  $u=p_{\mathcal L}(t,\,\cdot\,,z_2)$ in  $B^+(z_0,1)$  we get 
\begin{align*} 
		 \left|\nabla p_{\mathcal L} (t,z_0,z_2)\right|&\leq C\left(\|{\mathcal L}u\|_{L^p(B^+(z_0,2))}+\|u\|_{L^p(B^+(z_0,2))}\right)\\[1ex]
		 &\leq C\left(\|{\mathcal L}u\|_{L^\infty(B^+(z_0,2))}+\|u\|_{L^\infty(B^+(z_0,2))}\right).
\end{align*} 
 Using \eqref{upper estimates ver2 La} and  Proposition \ref{Time derivative estimates},  we get for suitable  $C, \kappa >0$ 
  \begin{align*} 
 	\|u\|_{L^\infty(B^+(z_0,2))}+\|\mathcal Lu\|_{L^\infty(B^+(z_0,2))}&\leq  C y_2^{-c}\left (y_2\wedge 1 \right)^{c}
 	\sup_{z\in B^+(z_0,2) }\exp\left(-\frac{|z-z_2|^2}{\kappa}\right).
 \end{align*}
Lemma \ref{lemma loc gauss}, combined with the previous inequalities, then implies (for suitable $C', \kappa'>0$)
 \begin{align*} 
 	 \left|\nabla p_{\mathcal L} (t,z_0,z_2)\right|&\leq C' y_2^{-c}\left (y_2\wedge 1 \right)^{c}
 	\exp\left(-\frac{|z_0-z_2|^2}{\kappa'}\right)
 \end{align*}
 which is the statement for $|t|=1$. The result for general $t$ follows by the scaling property \eqref{scaling 1}.
 
The analogous estimates $\nabla_{z_1} p_{\mathcal L^\ast}$  can be proved similarly using Remark \ref{os grad ajoint}. 
\qed 

The following corollary is a direct consequence of Theorem \ref{space derivative estimates}. For simplicity we state it only for $\mathcal L$.
\begin{cor}\label{Dy T(t)_b}
Let  $1<p < \infty$, $0< \frac{m+1}{p} <c+1$, $f\in L^p_{m}$.  Then for every  $t\in\Sigma_{\frac\pi 2-\arctan\frac{|a|}{1-|a|}- \epsilon}$ ,   $e^{t\mathcal L}f$ is differentiable in $\R^{N+1}_+$ and one has
	\begin{align}\label{der semigroup}
		 \nabla e^{t\mathcal L}f=\int_{\R^{N+1}_+}\nabla p_{\mathcal L}(t,\cdot,z_2)f(z_2)\,y_2^{c} dz_2.
	\end{align}
Moreover, for  every $\epsilon>0$, there exist $C,\ k>0$  such that,
	\begin{align}\label{der semigroup1}
 |\nabla e^{t\mathcal L}f(z_1)|\leq \frac{C}{ |t|^{\frac{N+2}{2}}} \int_{\R^{N+1}_+} \left(1\wedge \frac{y_2}{\sqrt{ |t|}}\right)^{c}\,\exp\left(-\frac{|z_1-z_2|^2}{k|t|}\right)|f(z_2)|\,dz_2.
\end{align}
 
	\end{cor}
{\sc Proof.} 
	Let $z_0, r>0$ such that $B(z_0,r)\in\R^{N+1}_+$. By Theorem \ref{space derivative estimates},  for almost every $z_1\in B(z_0,r)$, $z_2\in\R^{N+1}_+$, one has 
\begin{align}  \label{stimagradiente}
			\left|\nabla p_{\mathcal L} (t,z_1,z_2)\right|y_2^c\leq \frac{C}{ |t|^{\frac{N+2}{2}}} \left(1\wedge \frac{y_2}{\sqrt{ |t|}}\right)^{c}\,\exp\left(-\frac{|z_1-z_2|^2}{k_\epsilon|t|}\right)|f(z_2)|.
	\end{align}
	for  suitable  $C$ and $ \kappa$ depending also  on $r$ and $z_0$. Then \eqref{der semigroup} follows by differentiating under the integral sign since 
the right hand side of \eqref{stimagradiente} belongs to $L^{p'}_m$. Finally, \eqref{der semigroup1} is consequence of Theorem \ref{space derivative estimates}.
	\qed

\section{The semigroup in $L^1_c$ and the conservation property}\label{Section conservativity}
In this section we prove that the semigroups $e^{t\mathcal L}$, $e^{t\mathcal L^\ast }$, initially defined on $L^2_c$, extend also   on $L^1_c$ and we discuss the conservation property
\begin{align*}
	e^{t\mathcal L}1=1,\qquad e^{t\mathcal L^\ast }1=1,\qquad \quad \text{for any}\quad t\geq 0.
\end{align*}

Following the notation introduced in the appendix, we start by  rewriting the upper  estimates of Theorems \ref{true.kernel} and \ref{space derivative estimates}  in terms of  the family of integral operators
\begin{align*}
	S^{0,-c}(t)f(z_1)=t^{-\frac {N+1} 2}\int_{\R^{N+1}_+}  \left (\frac{y_2}{\sqrt t}\wedge 1 \right)^{c}
	\exp\left(-\frac{|z_1-z_2|^2}{\kappa t}\right)f(z_2) \,dz_2,
\end{align*}
defined for $t>0$ and $z_1=(x_1,y_1),\ z_2=(x_2,y_2)\in\R^{N+1}_+$. Here $\kappa$ is a positive constant; to shorten the notation we omit the dependence on $\kappa$  which may vary in each occurrence. We recall that, by Proposition \ref{Boundedness theta} and by \cite[Proposition A.1]{MNS-Caf-Schauder}, if $c+1>0$ then the family $S^{0,-c}(t)$ is uniformly bounded on $L^p_c$ for $1\leq p\leq\infty$.

\begin{prop}\label{Prop stime sem+grad family}
	Let $c+1>0$. Then, for some positive constant $C>0$ and  a  suitable $\kappa>0$, one has  for any  $f\in L^2_c$ and $t>0$
	\begin{align}
		\label{stime sem+grad family}
		\left|e^{t\mathcal L}f\right|+ \sqrt t\left|\nabla e^{t\mathcal L}f\right|	\leq CS^{0,-c}(t)|f|,\qquad 	\left|e^{t\mathcal L^\ast }f\right|+\sqrt t\left|\nabla e^{t\mathcal L^\ast }f\right|\leq C\,S^{0,-c}(t)|f|.
	\end{align}
\end{prop}
\begin{proof}
	The required estimates follows from    Theorems \ref{true.kernel},  Corollary \ref{Dy T(t)_b} and Remark \ref{oss equiv estimate}.
\end{proof}

We now prove that the semigroups $e^{t\mathcal L}$, $e^{t\mathcal L^\ast }$ extend from $L^2_c\cap L^1_c$ to  strongly continuous semigroups on $L^1_c$.

\begin{prop} \label{strong-cont}
	Let $c+1>0$ and $1\leq p<\infty$.	Then the  semigroup $e^{t\mathcal L}$  and  $e^{t\mathcal L^*}$ extend from $L^2_c\cap L^p_c$ to  strongly continuous semigroups on $L^p_c$.
\end{prop} 
\begin{proof}
	We prove the statement for $\mathcal L$; the analogous claim for $\mathcal L^\ast$ follows identically. By Proposition  \ref{generation L2}, $e^{t\mathcal L}$ is  contractive in $L^p_c$ for every $1\leq p\leq \infty$: then $e^{t\mathcal L}$  clearly extends, by density, from $L^2_c\cap L^p_c$ to  $L^p_c$ and the semigroup law is inherited  from the one of $L^2_c$.  Therefore we have only to prove the  strong continuity at $0$.  
	Assume first that $1<p<\infty$ and let $f, g \in C_c^\infty (\overline{\R^{N+1}_+})$. Then as $t \to 0$
	$$
	\int_{\R^{N+1}_+} (e^{t\mathcal L}f)\, g\, y^c dy\longrightarrow \int_{\R^{N+1}_+} fgy^cdy  ,
	$$
	by the strong continuity of $e^{t\mathcal L}$ in $L^2_c$. Since $e^{t\mathcal L}$ is contractive on $L^p_c$, a   density argument then proves that the previous limit holds for every $f \in L^p_c$, $g \in L^{p'}_c$. The semigroup is then weakly continuous, hence strongly continuous. Assume now $p=1$ and let us fix, preliminarily, $f\in L^1_c\cap L^2_c$. For any $r>0$ we set $B_r=B(0,r)\times ]0,r[$ and write $f=f_{|B_r}+f_{|B_r^c}$. Then we have, using the contractivity of $e^{t\mathcal L}$ on $L^1_c$,
	\begin{align}\label{strong-cont eq 1}
		\nonumber	\|e^{t\mathcal L}f-f\|_{L^1_c}&\leq \|e^{t\mathcal L}\left(f_{|B_r}\right)-f_{|B_r}\|_{L^1_c}+\|e^{t\mathcal L}\left(f_{|B_r^c}\right)-f_{|B_r^c}\|_{L^1_c}
		\\[1.5ex]&
		\leq \|e^{t\mathcal L}\left(f_{|B_r}\right)-f_{|B_r}\|_{L^1_c}+ 2 \|f_{|B_r^c}\|_{L^1_c}.
	\end{align}
	Recalling that  the measure $\mu=y^cdz$ is locally finite  and using  the H\"older inequality we get
	\begin{align*}
		\|e^{t\mathcal L}\left(f_{|B_r}\right)-f_{|B_r}\|_{L^1_c}&=\int_{B_r}\left|e^{t\mathcal L}\left(f_{|B_r}\right)-f_{|B_r}\right|\,y^cdz+\int_{B_r^c}\left|e^{t\mathcal L}\left(f_{|B_r}\right)\right|\,y^cdz\\[1ex]
		&\leq \left(\mu\left(B_r\right)\right)^{\frac 12}\left\|e^{t\mathcal L}\left(f_{|B_r}\right)-f_{|B_r}\right\|_{L^2_c}+\int_{B_r^c}\left|e^{t\mathcal L}\left(f_{|B_r}\right)\right|\,y^cdz\\[1ex]
		&:=A_1+A_2
	\end{align*}
	which goes to $0$ as $t\to 0$. Indeed  $\ds \lim_{t\to 0} A_1=0$ by the strong continuity of  $e^{t\mathcal L}$ in $L^2_c$. For the second term we observe that, using again  the contractivity of $e^{t\mathcal L}$ on $L^1_c$ 
	\begin{align*}
		\|f_{|B_r}\|_{L^1_c}\geq \|e^{t\mathcal L}\left(f_{|B_r}\right)\|_{L^1_c}=	\int_{B_r^c}\left|e^{t\mathcal L}\left(f_{|B_r}\right)\right|\,y^cdz+\int_{B_r}\left|e^{t\mathcal L}\left(f_{|B_r}\right)\right|\,y^cdz.
	\end{align*}
	Then taking the limit in the above estimate and observing that, from the previous point, 
	$$\int_{B_r}\left|e^{t\mathcal L}\left(f_{|B_r}\right)\right|\,y^cdz\to \|f_{|B_r}\|_{L^1_c},$$
	we get $\ds \lim_{t\to 0} A_2=0$. Then it follows from \eqref{strong-cont eq 1} that
	\begin{align*}
		\limsup_{t\to 0}\|e^{t\mathcal L}f-f\|_{L^1_c}\leq 2 \|f_{|B_r^c}\|_{L^1_c}\longrightarrow 0,\qquad \text{as}\quad r\to\infty
	\end{align*}
	which proves that 	$\ds \lim_{t\to 0}\|e^{t\mathcal L}f-f\|_{L^1_c}=0$ which is the required claim for   $f\in L^1_c\cap L^2_c$. The general case follows  by a   density argument.
\end{proof}

For the case $p=\infty$ of the above Proposition we refer the reader to \cite{MNS-Caf-Schauder} for   generation of semigroup, in a special case, in spaces of continuous functions.\\

We can now prove the conservation property of the semigroups.
\begin{prop}  \label{conservation}
	Let us assume $c+1>0$. Then the heat kernel $p_\mathcal L$ of $\mathcal L$ satisfies  
	\begin{align*}
		\int_{\R^{N+1}_+}p_{\mathcal L}(t,z_1,z_2) y_2^c\, dz_2=1
		\quad \forall \ t>0, \ z_1\in\R^{N+1}_+.
	\end{align*}
	The same conservation property also holds for the kernel $p_{\mathcal L^\ast}$ of the adjoint operator $\mathcal L^\ast$.
\end{prop}
{\sc Proof.}  
We prove the statement for $P_{\mathcal L}$; the analogous claim for $p_{\mathcal L^\ast}$ follows identically. We prove, equivalently, that for every $f\in L^2_c\cap L^1_c$ one has
\begin{equation*}
	\int_{\R^{N+1}_+}e^{t{\mathcal L}}f(z) y^c\, dz=\int_{\R^{N+1}_+}f(z) y^c\, dz
	\quad \forall \ t>0.
\end{equation*}
We fix a  cut-off function $\eta(z)=\vartheta(x)\theta(y)$ where $\vartheta\in C_c^{\infty}\left( B(0, 2)\right)$ is a  cut-off function in the $x$-variables which  satisfies
$0\le \vartheta\le 1$, $\vartheta=1$ in $B(0,1)\subset\R^N$ and  $\theta\in C_c^{\infty }([0,2[)$ such that  $0\leq \theta \leq 1$, $\theta(y)=1$ for $y\leq 1$. 
We also set, for every $n>0$, $\eta_n(z)=\eta\left(\frac{z}{n}\right)$. 
Then  differentiating  under the integral sign and recalling Section \ref{L2},  we have for every $t>0$
\begin{align*}
	I_n(t)&=\frac{d}{dt}\int_{\R^{N+1}_+}e^{t\mathcal L}f(z)\eta_n(z) y^c\, dz=\int_{\R^{N+1}_+}\mathcal L e^{t\mathcal L}f(z)\eta_n(z) y^c\, dz=-\mathfrak{a}(e^{t\mathcal L^*}f, \eta_n)\\
	&=-\int_{\R^{N+1}_+}\nabla e^{t\mathcal L}f(z)\cdot\nabla\eta_n(z) y^c\, dz-2\int_{\R^{N+1}_+}D_y e^{t\mathcal L}f(z)a\cdot\nabla_x\eta_n(z) y^c\, dz.
\end{align*} 
Therefore, since $|\nabla \eta_n|\leq \frac C n$, using Propositions \ref{Prop stime sem+grad family} and  \ref{Boundedness theta}  we get for some possibly different positive constants $C>0$
\begin{align*} 
	|I_n(t)|\leq\frac{C}{n}\|\nabla e^{t\mathcal L}f\|_{L^1_c}\leq \frac{C}{n\sqrt t}\|S^{0,-c}(t)|f|\|_{L^1_c}\leq \frac{C}{n\sqrt t}\|f\|_{L^1_c}
\end{align*}
which implies $\ds \lim_{n\to \infty}I_n(t)=0$. Then using Lebesgue's Theorem it follows that for every $t, s>0$
\begin{align*}
	&\int_{\R^{N+1}_+}e^{t\mathcal L}f(z)\ y^c\, dz-\int_{\R^{N+1}_+}e^{s\mathcal L}f(z) y^c\, dz\\&=\lim_{n\to\infty}\left(\int_{\R^{N+1}_+}e^{t\mathcal L}f(z)\eta_n(z) y^c\, dz-\int_{\R^{N+1}_+}e^{s\mathcal L}f(z)\eta_n(z) y^c\, dz\right)=\lim_{n\to\infty}\int_s^tI_n(r)\, dr=0
\end{align*}
and hence $$\int_{\R^{N+1}_+}e^{t\mathcal L}f(z)\ y^c\, dz=\int_{\R^{N+1}_+}e^{s\mathcal L}f(z) y^c\, dz,\qquad t,\ s>0.$$
Letting $s$ go to $0$, the claim follows by the strong continuity of  $e^{t{\mathcal L}}$ in $L^1_c$  proved in  Proposition \ref{strong-cont}.
\qed

\section{Kernel lower bounds}\label{section kernel lower}
This section, which contains the main result of the paper, is devoted to prove lower  estimates for the heat kernel of $\mathcal L$, defined in \eqref{La}, of the same form of the upper ones  of Theorem \ref{true.kernel}. We need some preparation.

We work in $\R^{N+1}_+$ endowed with the standard euclidean metric  and, for $z_0=(x_0,y_0)\in\R^{N+1}_+$ and $r>0$, we consider the cylindric balls 
$$Q(z_0,r):=B(x_0,r)\times [y_0,y_0+r)$$  in place of the standard balls $B(z_0,r)$; we keep the same notation to denote $Q(0,r):=B(0,r)\times [0,r[$. For $c+1>0$ we consider the weighted measure $\mu=y^c\,dx\,dy$ and we write $V (z_0,r):=\mu\left(Q(z_0,r)\right)$ to denote the measure  of the cylindric balls given by
\begin{align*}
	V (z_0,r)=\int_{Q(z_0,r)} y^c dz=w_Nr^{N}\int_{y_0}^{y_0+r} y^c dy.
\end{align*}
We need the following  elementary lemma which guarantees that $\mu$ is a doubling measure and which is proved in  {\cite[Lemma 5.2]{MNS-PerturbedBessel}}. 

\begin{lem}\label{Misura palle}
	Let $c+1>0$. 
	Then one has 
	\begin{align*}
		V(z_0,r)&=r^{N+1+c}\,V(\frac{z_0}r,1),\qquad 
		V(z_0,r)\simeq r^{N+1+c}\left(\frac{y_0}{r}\right)^{c}\left(\frac{y_0}{r}\wedge 1\right)^{-c}.
	\end{align*}
	In particular the function $V$ satisfies, for some  constants $C\geq 1$, the doubling condition
	\begin{align*}
		\frac{V(z_0,s)}{V(z_0,r)}\leq C \left(\frac{s}{r}\right)^N\left(1 \vee \frac{s}r\right)^{1+c^+},\qquad \forall\, s,  r>0.
	\end{align*}
\end{lem}

\begin{os} \label{doubling}
	The previous inequality  implies that, for some positive constant $C>0$, one has
	$$V(z_0,2r)\leq C V(z_0,r),\qquad \forall z_0\in\ R^{N+1}_+,\ r>0.$$
\end{os}

In order to prove lower heat kernel estimates we also need to rewrite the upper  bounds of Theorem \ref{true.kernel} in terms of  the measure of the cylindric balls.
\begin{prop}\label{kernel tilde} Let us assume $c+1>0$. Then the heat kernel $p_\mathcal L$ of $\mathcal L$, written with respect to the measure $y^c\ dx\,dy$, satisfies  
	\begin{align}\label{up kernel measure}
		p_{\mathcal L}(t,z_1,z_2)
		&\leq \frac{C}{V\left(z_1,\sqrt t\right)^{\frac 1 2}V\left(z_2,\sqrt t\right)^{\frac 1 2}}\exp\left(-\frac{|z_1-z_2|^2}{\kappa t}\right),\quad \forall \ t>0, \ z_1,\ z_2\in\R^{N+1}_+.
	\end{align}
\end{prop}
\begin{proof}
	The required estimates follows from    Theorems \ref{true.kernel} and Lemma \ref{Misura palle}.
\end{proof}
\medskip

In what follows we aim to prove that the similar lower bound
\begin{align}\label{low kernel measure}
	p_{\mathcal L}(t,z_1,z_2)
	&\geq \frac{C}{V\left(z_1,\sqrt t\right)^{\frac 1 2}V\left(z_2,\sqrt t\right)^{\frac 1 2}}\exp\left(-\frac{|z_1-z_2|^2}{\kappa t}\right),\quad \forall \ t>0, \ z_1,\ z_2\in\R^{N+1}_+
\end{align}
holds for the heat kernel of $\mathcal L$.\\

Before to prove this result, let us discuss, in an abstract setting, a  well established  strategy (see e.g. \cite[Chapter 7, Section 7.8]{Ouhabaz}) for deducing lower kernel estimates  for a non-negative self-adjoint operator $A$ acting  on an $L^2$-space $L^2\left(X,\mu\right)$; here $X$ is a homogeneous space i.e. a metric space $(X,d)$ endowed with a doubling measure $\mu$. Assume that the semigroup $e^{tA}$ has a kernel $p(t,z_1,z_2)$ which satisfies the conservation property as the one in Proposition \ref{conservation} and  some Gaussian upper bounds of the same form as in  \eqref{up kernel measure} (in this setting $V(z,r)$ is the measure of the $X$-balls and $|z_1-z_2|$ is the $X$-distance $d(z_1,z_2)$). Then the analogous lower bound can be derived from the conservation property of the semigroup, the Gaussian upper bound and the H\"older continuity of $p$ (or, a fortiori, from some gradient estimates on $p$). To be precise, the Gaussian lower bound   \eqref{low kernel measure} can be obtained by:
\begin{itemize}
	\item[(i)] proving, preliminarily, the  on-diagonal lower  estimates 
	\begin{align}\label{lower  diagonal}
		p(t,z,z)\geq \frac{C}{V(z,\sqrt t)},\qquad z\in X,\ t>0,
	\end{align}
	using the conservation  property of the kernel and the upper estimates \eqref{up kernel measure} (see e.g. \cite[Proposition 7.28]{Ouhabaz});
	\item[(ii)] proving the off-diagonal bound \eqref{low kernel measure}, adding also the Gaussian term, using the H\"older continuity of $p$ to go outside of the diagonal and an iterative argument based on the validity of the so-called chain condition  for $X$ (see e.g. \cite[Theorem 7.29]{Ouhabaz}).
\end{itemize}

In our case, due to the non self-adjointness of $\mathcal L$, this strategy does not work. Then, inspired by the method of Nash \cite{Nash,Fabes-Stroock}, we  prove  the on-diagonal lower estimate \eqref{lower  diagonal} by proving in Proposition \ref{log-kernel-est} the so-called  ``G-bound'', namely a lower bound of the integral of the logarithm of the heat kernel  with respect to a  weighted Gaussian measure and  whose validity relies on some suitable Poincar\'{e} inequalities. Some difficulties appear: in order to compensate the non self-adjointness of $\mathcal L$ we need to prove all the statements for both $\mathcal L$ and $\mathcal L^\ast$ and we need also to take control of the  weighted measure $y^cdz$ whose density can be either singular or degenerate at $y=0$ and at infinity. Then, thanks to the gradient estimates  proved in Section \ref{section gradient estimate}, we can go outside the diagonal proving the off-diagonal lower bound \eqref{low kernel measure}.\\

The proof is structured in many steps and requires some preliminary results.
\begin{os}\label{equivalent V up low}
	We point out that, by \eqref{equiv estimates up low} and Lemma \ref{Misura palle}     the  estimates 
	\begin{align*}
		p_{\mathcal L}(t,z_1,z_2)
		&\simeq \frac{C}{V\left(z_1,\sqrt t\right)^{\frac 1 2}V\left(z_2,\sqrt t\right)^{\frac 1 2}}\exp\left(-\frac{|z_1-z_2|^2}{\kappa t}\right)
	\end{align*}
	are equivalent, up to change the constants $C,k>0$,  to
	\begin{align*}
		p_{{\mathcal  L}}(t,z_1,z_2)
		\simeq \frac{C}{V\left(z_i,\sqrt t\right)}\,\exp\left(-\dfrac{|z_1-z_2|^2}{kt}\right),\qquad i=1,2.
	\end{align*}	
\end{os}

\smallskip

We start by proving, preliminarily, the lower estimate \eqref{low kernel measure} far away from the boundary i.e. for   $z_1=(x_1,y_1),\ z_2=(x_2,y_2)\in\R^{N+1}_+$  satisfying $\frac{y_1}{\sqrt{t}}\geq r,\ \frac{y_2}{\sqrt{t}}\geq r$.
Its proof follows  by domination with a uniformly elliptic  operator with bounded coefficients.

\begin{lem}\label{Lemma uniformly outside strip}
	Let $r>0$ and let us consider the operator  $\mathcal A=\Delta_x+a\cdot\nabla_xD_y+B_y$ acting in the $L^2$-space  $L^2\left(\R^{N}\times \left[r,+\infty\right[,y^cdz\right)$ endowed with Dirichlet boundary conditions. Then $\mathcal A$  generates an analytic   semigroup of integral operator which is positive for $t>0$ and  whose kernel $p_0(t,z_1,z_2)$, taken with respect to the Lebesgue measure, satisfies, for some positive constants $C,c>0$,
	\begin{equation}\label{fuori palla 1}
		p_0(t,z_1,z_2)\geq C\left(1\wedge \frac{y_1-r}{\sqrt{t}}\right)\left(1\wedge \frac{y_2-r}{\sqrt{t}}\right)t^{-\frac{N+1}{2}} \exp\left(-c\frac{|z_1-z_2|^2}{ t}\right)
	\end{equation} 
	for any  $t>0$, $y_1\geq r$, $y_2\geq r$.
\end{lem}
{\sc Proof.} 
We consider  in $L^2\left(\R^{N}\times \left[r,+\infty\right[,y^cdz\right)$  the  sesquilinear form 
\begin{align*}
	\mathfrak{a}_0(u,v)
	&:=
	\int_{\R^N \times \left[r,+\infty\right[}\langle \nabla u, \nabla \overline{v}\rangle\,y^{c} dx\,dy+2\int_{\R^N \times \left[r,+\infty\right[} D_yu\, a\cdot\nabla_x\overline{v}\,y^{c} dx\,dy , \\[1.5ex]
	D(\mathfrak{a}_0)&=H_{c,0}^1\left(\R^N \times \left[r,+\infty\right[\right):=\ov{C_c^\infty\left(\R^N \times \left]r,+\infty\right[\right)}
\end{align*}
where, in the latter equality, the closure is taken in the norm of $H^1_c$. Recalling Section \ref{L2}, $\mathfrak{a}_0$ is the restriction of $\mathfrak{a}$ to the space $H_{c,0}^1\left(\R^N \times \left[r,+\infty\right[\right)\hookrightarrow H^1_c$ and therefore inherits all its properties stated in Proposition \ref{prop-form}; moreover  $\mathcal A$ is the operator associated with $\mathfrak{a}_0$. The generation results for $\mathcal A$ then follows by classical results on forms as in \cite[Proposition 2.2]{Negro-Spina-SingularKernel}.
The existence of the kernel as well as its  estimate follows by \cite[Theorem 3.8]{cho}.\qed

\begin{prop} \label{kernel-exterior}
	Let $r>0$, $c+1>0$. Then there exists a positive constant $C=C(r)>0$ such that 
	\begin{align*}
		p_{\mathcal L}(t,z_1,z_2)\geq \frac{C}{V(z_1,\sqrt t)^\frac{1}{2}V(z_2,\sqrt t)^\frac{1}{2}}\exp\left(-c\frac{|z_1-z_2|^2}{ t}\right)
	\end{align*}
	for every $t>0$ and every $z_1,\ z_2\in\R^{N+1}_+$ satisfying $\frac{y_1}{\sqrt{t}}\geq r,\ \frac{y_2}{\sqrt{t}}\geq r$.
\end{prop}
{\sc Proof.} Let us suppose, for simplicity,  $r=1$. The general case follows similarly.
Let us set $\R^{N+1}_{\geq \frac 1 2}:=\R^{N}\times \left[\frac{1}{2},+\infty\right[$. We consider the uniformly elliptic operator $\mathcal A$ of Lemma \ref{Lemma uniformly outside strip}, in the case $r=\frac 1 2$,  acting in $L^2_c\left(\R^{N+1}_{\geq \frac 1 2}\right):=L^2\left(\R^{N+1}_{\geq \frac 1 2},y^cdz\right)$ and endowed with Dirichlet boundary conditions. By construction the operators $\mathcal A$ and $\mathcal L$ are associated respectively with the sesquilinear forms $\mathfrak{a}_0$, $\mathfrak{a}$.  Moreover 
\begin{align*}
	\mathfrak{a}_0=\mathfrak{a}_{|D(\mathfrak{a}_0)}\qquad\text{and}\qquad 	D(\mathfrak{a}_0)=H_{c,0}^1\left(\R^{N+1}_{\geq \frac 1 2}\right)\hookrightarrow H^1_c=D(\mathfrak{a}).
\end{align*} 
Then, under the notation of  \cite[Section 4]{Manavi-Vogt-Voigt}, the previous relations imply that $D(\mathfrak{a}_0)$ is an ideal of $D(\mathfrak{a})$  and moreover we have, in particular,
\begin{align*}
	\mathfrak{a}_0(u,v)&= \mathfrak{a}(u,v),\qquad \text{for any}\quad 0\leq u,v\in D(\mathfrak{a}_0).
\end{align*} 
Then, recalling that $\mathcal A$ and $\mathcal L$ generate positive semigroups, we  apply \cite[Corollary 4.2]{Manavi-Vogt-Voigt} thus obtaining  
\begin{align*}
	\left|e^{t\mathcal A}f\right|\leq e^{t\mathcal L}|f|,\qquad \forall f\in L^2_c\left(\R^{N+1}_{\geq \frac 1 2}\right)
\end{align*} 
which in terms of their heat kernels rewrites as
\begin{align*}
	p_{\mathcal L}(t,z_1,z_2)\geq p_0(t,z_1,z_2)y_2^{-c},\qquad t>0,\ z_1,z_2\in  \R^{N+1}_{\geq \frac 1 2},
\end{align*}
(note that $p_{\mathcal L}$ is written with respect to the measure $y^cdz$ and  $p_0$ with respect to the Lebesgue measure). Then specializing the previous inequality for $t=1$ and $z_1,z_2\in \R^N\times[1,+\infty[$ and using  \eqref{fuori palla 1},  we get for some positive constant $C,c>0$
$$p_{\mathcal L}(1,z_1,z_2)\geq p_0(1,z_1,z_2)y_2^{-c}\geq Cy_2^{-c}\exp\left(-c|z_1-z_2|^2\right).$$
Finally, using  the scaling property \eqref{scaling} and Lemma \ref{Misura palle}, we get 
\begin{align*}
	p_{\mathcal{L}}(t,z_1,z_2)&=t^{-\frac{N+1+c}2}p_{\mathcal{L}}\left(1,\frac{z_1}{\sqrt t},\frac{z_2}{\sqrt t}\right)\geq Ct^{-\frac{N+1+c}2}\left(\frac{y_2}{\sqrt t}\right)^{-c}\exp\left(-c\frac{|z_1-z_2|^2}{ t}\right)\\[1ex]&\simeq \frac{C}{V(z_2,\sqrt t)}\exp\left(-c\frac{|z_1-z_2|^2}{ t}\right)
\end{align*}
for every  $z_1$, $z_2\in\R^{N+1}_+$ such that $\frac{y_1}{\sqrt{t}}\geq 1$,  $\frac{y_2}{\sqrt{t}}\geq 1$. The required estimates then follows by Remark \ref{equivalent V up low}.
\qed

\smallskip

We now come to the main part of the section: we prove  the lower estimate \eqref{low kernel measure} near the boundary i.e. for   $z_1, z_2\in\R^{N+1}_+$  when at least one of $\frac{z_1}{\sqrt t}, \frac{z_2}{\sqrt t}$  lies in $\R^N\times ]0,1]$. The main point consists in Proposition \ref{log-kernel-est} where we prove the so-called  ``G-bound'', i.e. a lower bound of the integral of the logarithm of the heat kernel with respect to a  weighted Gaussian measure. We need some preparation.

\smallskip 

We fix  $\alpha>0$ and we consider the weighted  Gaussian measure $\nu :=y^ce^{-a|z|^2}dz$ on $\R^{N+1}_+$ which, from the assumptions  $\alpha, c+1>0$, is  finite on $\R^{N+1}_+$. By standard computation one has
\begin{align}\label{gauss-norm2}
	\nu\left(\R^{N+1}_+\right)=\int_{\R^{N+1}_+} y^c e^{-\alpha |z|^2}\,dz= a^{-\frac{N+1+c}{2}}\pi^{\frac N2}\Gamma\left(\frac{c+1}2\right)
\end{align} 
where $\Gamma$ is the Gamma function (see \cite[Section 3]{Negro-Spina-Poincare}).  We consider the space $L^2_\nu\left(\R^{N+1}_+\right):=L^2(\R^{N+1}_+, y^ce^{-a|z|^2}dz)$ and we define also the Sobolev space 
$$H^{1}_{\nu}(\R^{N+1}_+):=\{u \in L^2_{\nu}(\R^{N+1}_+) : \nabla u \in L^2_{\nu}(\R^{N+1}_+)\}$$
The following Poincar\'{e}-type inequality in $H^1_\nu(\R^{N+1}_+)$ is essential for the proof of Proposition \ref{log-kernel-est}.
\begin{teo}[Poincar\'{e} inequality] \label{poincare}
	Let $\alpha>0$,  $c+1>0$. Then,  for some positive constant $C>0$, one has
	\begin{align*}
		\left\|u-\overline u\right\|_{L^2_\nu(\R^{N+1}_+)}\leq  C \|\nabla u\|_{L^2_\nu (\R^{N+1}_+)},\qquad \forall u\in H^1_\nu(\R^{N+1}_+),
	\end{align*}
	where $\ds \overline u=\frac 1{\nu(\R^{N+1}_+)}\int_{\R^{N+1}_+} u\,d\nu(z)$.
\end{teo}
\begin{proof}
	See \cite[Theorem 3.7]{Negro-Spina-Poincare}.
\end{proof}
From now on, without any loss of generality,  we fix $\alpha>0$ in \eqref{gauss-norm2}  such that 
\begin{equation}\label{gauss-norm}
	\int_{\R^{N+1}_+} y^c\,e^{-\alpha|z|^2} dz=1.
\end{equation}
We define, for fixed $\theta \in [\frac{1}{2},1)$ and $z_2\in\R^{N+1}_+$, the function 
\begin{align*}
	u(t,z)&:=\theta p_{\mathcal L}(t,z,z_2)+1-\theta,\qquad \forall t>0,\; z\in\R^{N+1}_+.
\end{align*}
We collect in the next lemma the properties we need about $u$.
\begin{lem}\label{Prop u}
	Let $u$ be  the function above defined. Then the following properties hold.
	\begin{itemize}
		\item[(i)] $u(t,z)\geq 1-\theta>0$ for any $t>0$, $z\in\R^{N+1}_+$.
		\item[(ii)] For fixed $t>0$ one has $\frac{e^{-\alpha|z|^2}}u\in H^1_c(\R^{N+1}_+)$ and $\log u\in H^1_\nu(\R^{N+1}_+)$.
		\item[(iii)] If $z_2\in\R^N\times ]0,1]$ the one has  $$\sup_{t\in \left[\frac{1}{2},1\right]}u(t,z)\leq K,\qquad \forall z\in\R^{N+1}_+$$
		for some positive constant $K>0$, independent on $\theta$ and on $z_2$. 
		\item[(iv)]  One has    $$\lim_{R\to+\infty }\int_{Q(0,R)}u(t,z)y^c\,dz\geq \frac 1 2$$
		  uniformly  on $\theta$,  $t\in \left[\frac{1}{2},1\right]$ and  $z
		_2\in Q(0,1)=B(0,1)\times[0,1[$ .
	\end{itemize}
\end{lem}
\begin{proof}
	(i) follows from the positivity of $p_{\mathcal L}$.  (ii) follows from standard calculation, the previous property, the upper  estimates \eqref{upper estimates ver2} and of the gradient estimates of Theorem \ref{space derivative estimates}. To prove (iii) we use again  \eqref{upper estimates ver2} to write
	\begin{align*}
		p_{{\mathcal  L}}(t,z,z_2)
		\leq C t^{-\frac{N+1}{2}}  y_2^{-c} \left(1\wedge \frac{y_2}{\sqrt t}\right)^{c}\,\exp\left(-\dfrac{|z-z_2|^2}{kt}\right)
	\end{align*}
	which from the assumption $y_2\leq 1$  implies 
	\begin{align*}
		\sup_{t\in \left[\frac{1}{2},1\right]}u(t,z)\leq K \,\exp\left(-\dfrac{|z-z_2|^2}{k'}\right)\leq K
	\end{align*}
	for some positive constants $K,k'>0$.
	
	 Let us finally prove (iv). Since by the  definition $u(t,z)\geq \frac 1 2 \,p_{\mathcal L}(t,z,z_2)$ then using the conservation property of Proposition \ref{conservation} we get 
	\begin{align}\label{Prop u eq 1}
		\int_{Q(0,R)}  u(t,z) 
		\,y^{c} dz&\geq\frac 1 2\int_{Q(0,R)}  p_{\mathcal L}(t,z,z_2)\,y^{c} dz= \frac 1 2\left(1-\int_{Q(0,R)^c}  p_{\mathcal L}(t,z,z_2) \,y^{c} dz\right).
	\end{align}
	We then observe that,  by the  upper kernel estimates \eqref{upper estimates ver2} and the assumption $\frac 1 2 \leq t\leq 1$,  $|x_2|<1$, $y_2< 1$, one has,  for some positive constants $C,k>0$ which may be different in each occurrence,
	\begin{align*}
		p_{{\mathcal  L}}(t,z,z_2)
		&\leq C t^{-\frac{N+1}{2}}  y^{-c} \left(1\wedge \frac{y}{\sqrt t}\right)^{c}\,\exp\left(-\dfrac{|z-z_2|^2}{kt}\right)\\[1ex]
		&\leq C y^{-c} \left(1\wedge y\right)^{c}\,\exp\left(-\dfrac{|z|^2}{k}\right).
	\end{align*}
The last inequalities  then imply, by Lebesgue's Theorem,
\begin{align*}
	\int_{Q(0,R)^c}  p_{\mathcal L}(t,z,z_2) \,y^{c} dz\leq  	C \int_{Q(0,R)^c}  \left(1\wedge y\right)^{c}\,\exp\left(-\dfrac{|z|^2}{k}\right) dz\underset{R\to+\infty}{\longrightarrow} 0
\end{align*}
which using \eqref{Prop u eq 1} proves (iv).
\end{proof}
We can define now the Nash's $G$-function.
\begin{defi}\label{defi G-function}
	For fixed $\theta \in [\frac{1}{2},1)$ and $z_2\in\R^{N+1}_+$, we define 
	\begin{align*}
		G(t)&:=\int_{\R^{N+1}_+} \log u(t,z)e^{-\alpha|z|^2}y^c\, dz,\qquad \forall t>0.
	\end{align*}
	We remark that, in virtue of Theorem \ref{true.kernel} and of \eqref{upper estimates ver2},  $G(t)$ is well defined and finite.
\end{defi}
We will estimate $G(1)$ from below trough a differential inequality satisfied by $G(t)$ which we prove in the following lemma.
\begin{lem}\label{diff ineq G}
	The following properties hold.
	\begin{itemize}
		\item[(i)] $G(t)\leq 0$ for any t$>0$. 
		\item[(ii)] For some positive constants $A, B>0$, independent of $t$, $z_2$ and $\theta$, one has 
		\begin{equation} \label{increasing}
			G'(t)\geq -A+B\int_{\R^{N+1}_+}  |\nabla \log u|^2 e^{-\alpha|z|^2}\,y^{c} dz.
		\end{equation}
		In particular  $G(t)+At$ is increasing.
	\end{itemize}
\end{lem}
\begin{proof}
	To prove (i) we recall that, by Proposition \ref{conservation} one has  \begin{align*}
		\int_{\R^{N+1}_+}  p_{\mathcal L}(t,z,z_2)y^c\, dz=1.
	\end{align*}
	Recalling \eqref{gauss-norm} and applying  Jensen's inequality   it follows that
	\begin{align*}
		G(t)&=\int_{\R^{N+1}_+} \log u(t,z)e^{-\alpha|z|^2}y^c\, dz\leq 	\log \left(\int_{\R^{N+1}_+}  u(t,z)e^{-\alpha|z|^2}y^c\, dz_2\right)\\&=\log \left(\int_{\R^{N+1}_+}  \theta p_{\mathcal L}(t,z,z_2)e^{-\alpha|z|^2}y^c\, dz+\int_{\R^{N+1}_+} (1-\theta) e^{-\alpha|z|^2}y^c\, dz\right)\\&\leq \log\left(\theta\int_{\R^{N+1}_+}   p_{\mathcal L}(t,z,z_2) y^c\, dz+(1-\theta)\right)=\log 1=0.
	\end{align*}
	To prove (ii) we observe preliminarily that 
	\begin{align*}
		\partial_t u(t,z)=\theta \mathcal L p_{\mathcal L}(t,z,z_2)=\mathcal L u(t,z).
	\end{align*} 
	Then  differentiating under the integral sign and using  Lemma \ref{Prop u}  we   obtain
	\begin{align*}
		G'(t)&=\int_{\R^{N+1}_+} \partial_t\log u(t,z)e^{-\alpha|z|^2}y^c\, dz\\[1ex]
		&=\int_{\R^{N+1}_+} \mathcal Lu(t,z)\frac{e^{-\alpha|z|^2}}{u}y^c\, dz=-\mathfrak{a}\left( u,\frac{e^{-\alpha|z|^2}}{u}\right).
	\end{align*}
	Then recalling \eqref{form a def} and by standard calculation we get
	\begin{align*}
		G'(t)&=-\int_{\R^{N+1}_+}  \nabla u\cdot\nabla \left(\frac{e^{-\alpha|z|^2}}{u}\right)y^{c}\, dz-2\int_{\R^{N+1}_+}  D_y  u\,a\cdot\nabla_x\left(\frac{e^{-\alpha|z|^2}}{u}\right)\,y^{c}\, dz\\&
		=2\alpha\int_{\R^{N+1}_+}  \nabla \log u\cdot z\, e^{-\alpha|z|^2}\,y^{c}\, dz+\int_{\R^{N+1}_+}  |\nabla \log u|^2e^{-\alpha|z|^2}\,y^{c}\, dz\\[1ex]
		&+4\alpha\int_{\R^{N+1}_+}  D_y\log u\,  a\cdot x e^{-\alpha|z|^2}\,y^{c}\, dz +2\int_{\R^{N+1}_+} a\cdot \nabla_x \log u\,D_y\log u\, e^{-\alpha|z|^2}\,y^{c}\, dz.
	\end{align*}
	
	By Young's inequality we get for any  $\eps>0$ and some positive constant $C>0$
	\begin{align*}
		\left|2\alpha\int_{\R^{N+1}_+}  \nabla \log u\cdot z\, e^{-\alpha|z|^2}\,y^{c}\, dz\right|&\leq \eps\int_{\R^{N+1}_+}  |\nabla \log u|^2\, e^{-\alpha|z|^2}\,y^{c}\, dz\\&+C \int_{\R^{N+1}_+}  |z|^2\, e^{-\alpha|z|^2}\,y^{c}\, dz,
	\end{align*}
	\begin{align*}
		\left|4\alpha\int_{\R^{N+1}_+} D_y\log u\,  a\cdot x e^{-\alpha|z|^2}\,y^{c}\, dz\right|&\leq \eps\int_{\R^{N+1}_+}  |D_y \log u|^2\, e^{-\alpha|z|^2}\,y^{c}\, dz\\&+C \int_{\R^{N+1}_+}  |x|^2\, e^{-\alpha|z|^2}\,y^{c}\, dz,
	\end{align*}
	\begin{align*}
		\left|2\int_{\R^{N+1}_+} a\cdot \nabla_x \log u\,D_y\log u\, e^{-\alpha|z|^2}\,y^{c}\, dz\right|&\leq |a|\int_{\R^{N+1}_+}  |\nabla_x \log u|^2\, e^{-\alpha|z|^2}\,y^{c}\, dz\\&+|a|\int_{\R^{N+1}_+}  |D_y \log u|^2\, e^{-\alpha|z|^2}\,y^{c}\, dz.
	\end{align*}
	Putting together the previous inequality and choosing $\epsilon$ sufficiently small we get
	\begin{equation*}
		G'(t)\geq -A+B\int_{\R^{N+1}_+}  |\nabla \log u|^2 e^{-\alpha|z|^2}\,y^{c} dz,
	\end{equation*}
	some positive constants $A, B>0$, independent of $t$, $z_2$ and $\theta$. This proves the required claim.
\end{proof}
\begin{os}\label{os G-ast}
	We can also consider, similarly, 
	\begin{align*}
		u^\ast(t,z)&:=\theta p_{\mathcal L^\ast }(t,z,z_2)+1-\theta,\qquad 
		G^\ast (t):=\int_{\R^{N+1}_+} \log u^\ast (t,z)e^{-\alpha|z|^2}y^c\, dz.
	\end{align*}
	Then, adapting almost identically all the  proofs to the operator $\mathcal L^\ast$,   the same results of Lemmas \ref{Prop u} and \ref{diff ineq G} keep to hold for $u^\ast$ and $G^\ast$.
\end{os}

\begin{prop}[G-bound] \label{log-kernel-est}
	There exists a positive constant $C>0$ such that   one has
	\begin{align*}
		\int_{\R^{N+1}_+} \log p_{\mathcal L}(1,z,z_2)e^{-\alpha|z_2|^2}y_2^c\, dz_2\geq -C,\quad 	\int_{\R^{N+1}_+} \log p_{\mathcal L^*}(1,z,z_2)e^{-\alpha|z_2|^2}y_2^c\, dz_2\geq -C.
	\end{align*}
	for any $z\in B(0,1)\times]0,1[$.
\end{prop}
\begin{proof} Let $\theta \in [\frac{1}{2},1)$, $z_2\in B(0,1)\times]0,1[$ and let us consider, recalling Definition \ref{defi G-function} , the function 
\begin{align*}
	G(t)&=\int_{\R^{N+1}_+} \log u(t,z)e^{-\alpha|z|^2}y^c\, dz,\qquad \forall t\in \left[\frac 12,1\right].
\end{align*} 
Our aim is to prove that 
\begin{align}\label{gbound eq1}
	G(1)\geq -C
\end{align} 
for some positive constant $C>0$ independent of $\theta$ and of $z_2$. Indeed once proved \eqref{gbound eq1}, then  taking the limit for $\theta\to 1$ and recalling that $p_{\mathcal L}(1,z,z_2)=p_{\mathcal L^*}(1,z_2,z)$, one obtains
\begin{align*}
	\int_{\R^{N+1}_+} \log p_{\mathcal L}(t,z,z_2)e^{-\alpha|z|^2}y^c\, dz=\int_{\R^{N+1}_+} \log p_{\mathcal L^*}(1,z_2,z)e^{-\alpha|z|^2}y^c\, dz\geq -C
\end{align*}
which is the second required inequality for $p_{\mathcal L^*}$. The analogue inequality for  $p_{\mathcal L}$  follows in the same manner after replacing the function $G$ with $G^\ast$ and using Remark \ref{os G-ast}.

Let us then prove \eqref{gbound eq1}.  Recalling that, by (ii) of Lemma \ref{Prop u},   $\log u\in H^1_\nu(\R^{N+1}_+)$,  we can combine  \eqref{increasing} and  the weighted Poincar\'e inequality of Theorem \ref{poincare} applied to $\log u$  to obtain  
$$G'(t)\geq -A+B\int_{\R^{N+1}_+}  (\log u-G(t))^2 e^{-\alpha|z|^2}\,y^{c} dz$$
for some positive constants $A, B>0$, independent of $t$, $z_2$ and $\theta$. In particular the function   $G(t)+At$ is increasing.

We make, preliminarily, the following useful observation: the previous properties allow, without any loss of generality, to suppose the function  $G$ very negative, i.e. satisfying
\begin{align}\label{gbound eq2}
	G(t)<-D,\qquad \forall t\in \left[\frac 12,1\right],
\end{align}
for some  arbitrary large constant $D>0$ to be specified later. Indeed if \eqref{gbound eq2} fails  for some $t_0\in \left[\frac 12,1\right]$ then  we  have, from the monotonicity of $G(t)+At$, 
$$G(1)+A\geq G(t_0)+At_0\geq -D+\frac A 2$$
which already   imply the required claim \eqref{gbound eq1}.

 Assuming   \eqref{gbound eq2}  we are going now to estimate $G(1)$ from below by proving the differential inequality \eqref{Gbound eq 4}  which we derive using Lemma \ref{diff ineq G} and the  upper kernel estimates of Theorem \ref{true.kernel}.
We start by observing that by (iii) of  Lemma \ref{Prop u} one has
\begin{align}\label{gbound eq3}
\sup_{t\in \left[\frac{1}{2},1\right]}u(t,z)\leq  K,\qquad \forall  z\in\R^{N+1}_+
\end{align}
for some positive constant $K$ independent on $\theta$ and $z_2$.
Moreover the function $u\mapsto \frac{(\log u-G(t))^2}{u}$ is decreasing as a function of $u$ in $[e^{2+G(t)},+\infty[$. Therefore, choosing in \eqref{gbound eq2} $D$ sufficiently large,  we can then suppose $e^{2+G(t)}< K$  thus obtaining, since by  \eqref{gbound eq3}  $u\leq k$,
\begin{align*}  
	G'(t)&\geq -A+B\int_{\R^{N+1}_+}  (\log u-G(t))^2 e^{-\alpha|z|^2}\,y^{c} dz\\[1ex]
	&\geq -A+B\int_{\{u\geq e^{2+G(t)}\}}  \frac{(\log u-G(t))^2}{u}u\, e^{-\alpha|z|^2}\,y^{c} dz\\[1.5ex]
	&\geq -A+B\frac{(\log K-G(t))^2}{K}\int_{\{u\geq e^{2+G(t)}\}}  u(t,z) e^{-\alpha|z|^2}\,y_2^{c} dz_2.
\end{align*}
Next we prove that $\ds \int_{\{u\geq e^{2+G(t)}\}}  u(t,z) e^{-\alpha|z|^2}\,y^{c} dz$ is bounded from below. Using (iv) of  Lemma \ref{Prop u}  we get, choosing a sufficiently large $R>0$, 
$$\int_{Q(0,R)}  u(t,z)\,y^{c} dz\geq \frac 1 4$$ which implies
\begin{align*}
	\int_{u\geq e^{2+G(t)}}  u(t,z) e^{-\alpha|z|^2}\,y^{c} dz&=	\int_{\R^{N+1}_+}  u(t,z) e^{-\alpha|z|^2}\,y^{c} dz-	\int_{\{u< e^{2+G(t)}\}}  u(t,z) e^{-\alpha|z|^2}\,y^{c} dz\\[1ex]
	&\geq e^{-\alpha R^2}\int_{Q(0,R)}  u(t,z)\,y^{c} dz-
	e^{2+G(t)}\\[1ex]&
	\geq \frac 1 4 e^{-\alpha R^2}-
	e^{2+G(t)}.
\end{align*}
We then obtain
\begin{align*} 
	G'(t)&\geq -A+B\frac{(\log K-G(t))^2}{K}\left( \frac{1}{4}e^{-\alpha R^2}-e^{2+G(t)}\right).
\end{align*}

Choosing $D$ large enough in \eqref{gbound eq2} in  order to have $ \frac{1}{4}e^{-\alpha R^2}-e^{2+G(t)}>0$ and using  Young's inequality 
$$2|\log k||G(t)|<\epsilon G(t)^2+\frac{|\log k|^2}{\epsilon}$$
 with $\epsilon=\frac 1 2 $, we get for some possibly different constants $A, B>0$
\begin{align*}  
	G'(t)&\geq -A+B\, G(t)^2.
\end{align*}
Choosing again  $D$ large enough in \eqref{gbound eq2} in  order to have $G(t)^2>2\frac A B$, the previous inequality implies
\begin{align}\label{Gbound eq 4}  
	G'(t)&\geq \frac{B}{2} G(t)^2.
\end{align}
Integrating between $\frac 1 2$ and $1$ and recalling that $G(t)<0$, we finally obtain $G(1)\geq -\frac{4}{B}$ which is the required claim.
\end{proof}

We can now deduce, from  the G-bound of  Proposition \ref{log-kernel-est}, the  lower estimates \eqref{lower  diagonal} on the diagonal.

\begin{prop} \label{lower-diag}
There exists a positive constant $C>0$ such that 
	\begin{align*}
		p_{\mathcal L}(t,z,z)\geq \frac{C}{V(z,\sqrt t)},
		\qquad \forall \ t>0, \ z\in\R^{N+1}_+.
	\end{align*}
\end{prop}
{\sc Proof.} Recalling the scaling  properties \eqref{scaling} of the kernel 
 and of the function $V$ of Lemma \ref{Misura palle}, we have
 \begin{align*}
	p_{\mathcal{L}}(t,z,z)&=\left(\frac t2\right)^{-\frac{N+1+c}2}p_{\mathcal{L}}\left(2,\sqrt 2\frac{z}{\sqrt t},\sqrt 2\frac{z}{\sqrt t}\right),\\[1ex]
		V(z,\sqrt t)&=\left(\frac t2\right)^{-\frac{N+1+c}2}\,V\left(\sqrt 2\frac{z}{\sqrt t},\sqrt 2\right).
\end{align*}
From the previous equalities we can then assume, without any loss of generality, $t=2$. 

If $y\geq 1$, the claimed bound    follows from Proposition \ref{kernel-exterior}. Let us assume, therefore,   $y<1$. By using the   invariance for translation in the $x$-variable of the kernel \eqref{invariance x-translation} and of the function $V$ of Lemma \ref{Misura palle} we have
\begin{align*}
	p_{\mathcal{L}}(t,z+x_0,z+x_0)=	p_{\mathcal{L}}(t,z,z),\qquad 	V(z+x_0,\sqrt t)&=	V(z,\sqrt t)
\end{align*}
where for any $z=(x,y)\in\R^{N+1}_+$ and any $x_0\in\R^{N}$, we wrote,  with a little abuse of notation,  $z+x_0=\left(x+x_0,y\right)$ to denote the translation in the $x_0$-direction. Choosing then  a suitable $x_0\in\R^N$,  we may also assume  $|x|< 1$. 
Let $\alpha>0$ satisfying (\ref{gauss-norm}); then by using  the reproducing property of the kernel we get
\begin{align*}
	p_{\mathcal L}(2,z,z)&=\int_{\R^{N+1}_+}p_{\mathcal L}(1,z,z_2) p_{\mathcal L}(1,z_2,z)y_2^c\, dz_2\\[1ex]
	&=\int_{\R^{N+1}_+}p_{\mathcal L}(1,z,z_2)p^*_{\mathcal L}(1,z,z_2)y_2^c\, dz_2\\[1ex]&\geq \int_{\R^{N+1}_+}p_{\mathcal L}(1,z,z_2)p^*_{\mathcal L}(1,z,z_2)e^{-\alpha|z_2|^2}y_2^c\, dz_2.
\end{align*}
Applying  then Jensen inequality  with respect to the unitary measure $\nu=e^{-\alpha|z_2|^2}y_2^cdz_2$ and using Proposition \ref{log-kernel-est},  it follows that, for some positive constant $C>0$, one has 
\begin{align*}
	 \log(p_{\mathcal L}(2,z,z))\geq& \int_{\R^{N+1}_+}\log\Big(p_{\mathcal L}(1,z,z_2)p^*_{\mathcal L}(1,z,z_2)\Big)e^{-\alpha|z_2|^2}y_2^c\, dz_2\\[1ex]
	 =&\int_{\R^{N+1}_+}\log p_{\mathcal L}(1,z,z_2)e^{-\alpha|z_2|^2}y_2^c\, dz_2\\[1ex]
	 &\hspace{10ex}+\int_{\R^{N+1}_+}\log p^*_{\mathcal L}(1,z,z_2)e^{-\alpha|z_2|^2}y_2^c\, dz_2\geq -C
\end{align*}
which implies
\begin{align*}
	p_{\mathcal L}(2,z,z)\geq  e^{-C}.
\end{align*}
Recalling  Lemma \ref{Misura palle},  since we are assuming  $y<1$, we have 
$$V(z,\sqrt 2)\simeq V(z,1)\simeq y^{c}\left(y\wedge 1\right)^{-c}=1;$$
 the above inequality  is then equivalent to
\begin{align*}
	p_{\mathcal L}(2,z,z)\geq   \frac{C}{V(z,\sqrt 2)}.
\end{align*}
which proves the required claim. 
\qed

We may now use  the gradient estimates  proved in Section \ref{section gradient estimate}, to extend the lower bounds of Proposition \ref{lower-diag} ``near the diagonal'' i.e. in the range  $|z_1-z_2|<c\sqrt t$, for some $c>0$.
\begin{prop} \label{lower-small-ball}
There exist some positive constants $C,c>0$ such that 
	\begin{align*}
		p_{\mathcal L}(t,z_1,z_2)\geq \frac{C}{V(z_2,\sqrt t)},
		\qquad \forall \ t>0, \ z_1, z_2\in\R^{N+1}_+\quad \text{ s.t.} \quad |z_1-z_2|\leq c\sqrt{t}.
	\end{align*}
\end{prop}
{\sc Proof.}
Using the mean value theorem we get 
\begin{align*}
	|p_{\mathcal L}(t,z_1,z_2)-p_{\mathcal L}(t,z_1',z_2)|\leq |\nabla p_{\mathcal L}(t,\tilde{z},z_2)||z_1-z_1'|,\qquad \forall z_1,\ z_1',\ z_2\in\R^{N+1}_+
\end{align*}
where $\tilde z$ lies in the segment linking $z_1$ and $z_1'$. Then using  the gradient estimate of Theorem \ref{space derivative estimates} and Remark \ref{equivalent V up low} we get
	\begin{align*}
	|\nabla p_{{\mathcal  L}}(t,\tilde z,z_2)|\leq 
	\frac{C}{\sqrt t\,V\left(z_2,\sqrt t\right)}\,\exp\left(-\dfrac{|z_1-z_2|^2}{kt}\right)\leq 	\frac{C}{\sqrt t\,V\left(z_2,\sqrt t\right)},
\end{align*}
which combined with the previous inequality yields
\begin{align*}
	|p_{\mathcal L}(t,z_1,z_2)-p_{\mathcal L}(t,z_1',z_2)|\leq \frac{C}{t^\frac{1}{2}V(z_2,\sqrt t)} |z_1-z_1'|.
\end{align*}
 Then, specializing  the above inequality to $z_1'=z_2$ and  using   Proposition \ref{lower-diag}, we get  for some positive constants $C,\ C_1>0$.
\begin{align*}
	p_{\mathcal L}(t,z_1,z_2)&\geq p_{\mathcal L}(t,z_2,z_2)-\frac{C}{t^{\frac{1}{2}} V(z_2,\sqrt t)} |z_1-z_2|\\[1ex]&\geq \frac{C_1}{V(z_2,\sqrt t)}-\frac{C}{t^{\frac{1}{2}}V(z_2,\sqrt t)} |z_1-z_2|
	= \frac{{C}}{V(z_2,\sqrt t)}\left(c-t^{-\frac{1}{2}} |z_1-z_2|\right)
\end{align*}
where we set $c=C_1/C$. Then imposing  $c-t^{-\frac{1}{2}}|z_1-z_2|\geq \frac c 2$, that is  $\frac{|z_1-z_2|}{\sqrt t}\leq \frac{c}{2}$, we get the required claim.\qed 

We can finally deduce the lower  estimate \eqref{low kernel measure} in the whole space using a classical iterative  argument based on  the doubling property of Lemma \ref{Misura palle} satisfied by the function $V$  and  the chain condition satisfied by $\R^{N+1}_+$, namely the existence, for any $z_1$, $z_2\in\R^{N+1}_+$ of a sequence of point $\omega_0=z_1,\ \omega_2,\ \dots,\ \omega_n=z_2\in\R^{N+1}_+$ joining $z_1$ and $z_2$ and satisfying
\begin{align*}
	|\omega_i-\omega_{i+1}|\leq C\,\frac{|z_1-z_2|}n,\qquad \forall i=0,\dots,n-1.
\end{align*}
\begin{teo} \label{lower}
There exists a positive constant $C$ such that the heat kernel $p_\mathcal L$ of $\mathcal L$ satisfies  
	\begin{align*}
		p_{\mathcal L}(t,z_1,z_2)
		&\geq \frac{C}{V\left(z_1,\sqrt t\right)^{\frac 1 2}V\left(z_2,\sqrt t\right)^{\frac 1 2}}\exp\left(-\frac{|z_1-z_2|^2}{\kappa t}\right),\quad \forall \ t>0, \ z_1,\ z_2\in\R^{N+1}_+.
	\end{align*}
\end{teo}
{\sc Proof.} The proof follows identically as in the proof of  \cite[Theorem 7.29]{Ouhabaz}.\qed
 
 \medskip

 We end the section by  stating the final result for the  general operator \eqref{general operator def} of Section \ref{section general operator} under  Neumann or oblique derivative boundary condition which is a direct consequence of  Theorems \ref{true.kernel}  and  \ref{lower}   and   the transformation explained in  Section \ref{Section oblique}.
\begin{teo} \label{complete-sharp}
	Let $v=(d,c)\in\R^{N+1}$ with $d=0$ if $c =0$, $A$ be 
	an  elliptic matrix and let us consider the operator
	\begin{align*}
		\mathcal L =\mbox{Tr }\left(AD^2u\right)+\frac{ v\cdot \nabla }{y}
	\end{align*} endowed with the boundary conditions
\begin{align*}
	&\lim_{y \to 0} y^{\frac c \gamma}\, v \cdot \nabla u=0,\qquad \text{(if $c\neq 0$)},\qquad \qquad 
	\lim_{y \to 0} y^{\frac c \gamma}\, D_y u=0,\qquad \hspace{2ex}\text{(if $c=0$)}
\end{align*}
where $\gamma=a_{N+1,N+1}$. If   $\frac{c}{\gamma}+1>0$ then the semigroup $(e^{t{\mathcal L}})_{t> 0}$ consists of integral operators	and its heat kernel $p_{{\mathcal L}}$, written  with respect the measure $y^\frac{c}{\gamma}dz$, satisfies for some  $C,k>0$, 
 \begin{align*}
 	p_{{\mathcal  L}}(t,z_1,z_2)
 	\simeq C t^{-\frac{N+1}{2}} y_1^{-\frac{c}{2\gamma}} \left(1\wedge \frac {y_1}{\sqrt t}\right)^{\frac{c}{2\gamma}} y_2^{-\frac{c}{2\gamma}} \left(1\wedge \frac{y_2}{\sqrt t}\right)^{\frac{c}{2\gamma}}\,\exp\left(-\dfrac{|z_1-z_2|^2}{kt}\right).
 \end{align*}
\end{teo}

\section{Appendix}\label{Appendix}

We consider on the weighted spaces $L^p_m(\R^{N+M}):=L^p\left(\R^{N+M},|y|^m,dxdy\right)$ the two-parameters  family of integral operators $\left(S_{\alpha,\beta}(t)\right)_{t>0}$  defined for $\alpha,\beta\in\R$  and $t>0$ by
\begin{align*}
	S^{\alpha,\beta}(t)f(z_1)=t^{-\frac {N+M} 2}\,\left (\frac{|y_1|}{\sqrt t}\wedge 1 \right)^{-\alpha}\int_{\R^{N+M}}  \left (\frac{|y_2|}{\sqrt t}\wedge 1 \right)^{-\beta}
	\exp\left(-\frac{|z_1-z_2|^2}{\kappa t}\right)f(z_2) \,dz_2,
\end{align*}
where $\kappa$ is a positive constant, $z_1=(x_1,y_1),\ z_2=(x_2,y_2)\in\R^{N+M}$. We omit the dependence on $\kappa$ even though in some proofs we need to vary it.
We characterize in the proposition below the  boundedness of $\left(S_{\alpha,\beta}(t)\right)_{t>0}$ on $L^p_m$. The proof follows by using the strategy in \cite[Proposition 12.2]{MNS-Caffarelli} where the case $1<p<\infty$ has  been considered. For completeness we refer to \cite[Proposition 6.2]{MNS-Caf-Schauder} for the case $p=\infty$.\\

We start by observing  that the  scale homogeneity of $S^{\alpha,\beta}$ is $2$ since a simple change of variable in the integral yields
\begin{align*}
	S^{\alpha,\beta}(t)\left(I_s f\right)=I_s \left(S^{\alpha,\beta}(s^2t) f\right),\qquad I_sf(z)=f(sz),\qquad t, s>0,
\end{align*}
which in particular gives 
\begin{align*}
	S^{\alpha,\beta}(t)f =I_{1/\sqrt t} \left(S^{\alpha,\beta}(1) I_{\sqrt t} f\right),\qquad t> 0.
\end{align*}
The boundedness of $S^{\alpha,\beta}(t)f $ in $L^p_m(\R^{N+M})$ is then equivalent to that for $t=1$ and $\|S^{\alpha,\beta}(t)\|_p= \|S^{\alpha,\beta}(1)\|_p$.

\begin{prop}\label{Boundedness theta}
	Let $m\in\R$, $\theta\geq 0$ and $1\leq p<\infty$. The following properties are equivalent.
	\begin{itemize}
		\item[(i)]  For every $t>0$ $(S^{\alpha,\beta}(t))_{t \ge 0}$ is  bounded from   $L^p_m$ to $ L^p_{m-p\theta}$ and
		\begin{align*}
			\|S^{\alpha,\beta}(t)\|_{L^p_m\to L^p_{m-p\theta}}\leq Ct^{-\frac\theta 2}.
		\end{align*}
		\item[(ii)] $S^{\alpha,\beta}(1)$ is  bounded from  from   $L^p_m$ to $ L^p_{m-p\theta}$.
		\item[(iii)] One of the following properties:
		\begin{enumerate}
			\item $1<p<\infty$ and $\alpha+\theta <\frac{M+m}p < M-\beta$;\smallskip 
			\item $p=1$ and $\alpha+\theta <M+m \leq M-\beta$.
		\end{enumerate}
		In  particular $\alpha+\beta<M-\theta$.
	\end{itemize}
\end{prop}
\begin{proof}
	The proof follows identically  as in  \cite[Proposition 12.2]{MNS-Caffarelli} where the case $N=0$ has been considered. 
\end{proof}

We refer also the reader to \cite[Proposition A.1]{MNS-Caf-Schauder} for  the case $p=\infty$ of the above Proposition.

\bibliography{../TexBibliografiaUnica/References}
\end{document}